\documentclass{article}

\usepackage{amsmath,amssymb}
\usepackage{amsthm}
\usepackage{enumitem}
\usepackage{multicol}
\usepackage{graphicx}
\usepackage{url}
\usepackage[colorlinks]{hyperref}

\usepackage{color}

\usepackage{listings}
\usepackage{tikz}
\tikzset{
every node/.style={circle, draw, inner sep=2pt},
every picture/.style={thick}
}
\usetikzlibrary{matrix,arrows,calc}

\newtheorem{theorem}{Theorem}
\newtheorem{lemma}[theorem]{Lemma}
\newtheorem{proposition}[theorem]{Proposition}
\newtheorem{corollary}[theorem]{Corollary}

\theoremstyle{definition}
\newtheorem{definition}[theorem]{Definition}
\newtheorem{observation}[theorem]{Observation}
\newtheorem{remark}[theorem]{Remark}
\newtheorem{example}[theorem]{Example}
\newtheorem{problem}{Problem}
\newtheorem{question}[theorem]{Question}

\newenvironment{obs}{\begin{observation}\bgroup\rm }{\egroup\end{observation}}

\newcommand{\F}{\mathcal{F}}

\newcommand{\diag}{\operatorname{diag}}
\newcommand{\minors}{\operatorname{minors}}

\newcommand{\Pf}{P_4}
\newcommand{\Gaa}{\sf fork}
\newcommand{\Gab}{\sf 4\text{-}pan}
\newcommand{\Gac}{\sf bull}
\newcommand{\Gad}{\sf dart}
\newcommand{\Gae}{\sf P_5}
\newcommand{\Gaf}{\sf co\text{-}4\text{-}pan}
\newcommand{\Gag}{\sf 3\text{-}fan}
\newcommand{\Gah}{\sf kite}
\newcommand{\Gai}{\sf S_6 +e}
\newcommand{\Gaj}{\sf \overline{diamond+K_2}}
\newcommand{\Gak}{\sf K_{3,3}+e}
\newcommand{\Gal}{\sf \overline{P_3 + \overline{P_3}}}
\newcommand{\Gam}{\sf K_{1,1,1,2,2}}
\newcommand{\Gan}{\sf K_{1,1,1,1,4}}

\begin{document}


\title{Graphs with few trivial characteristic ideals}


\author{Carlos A. Alfaro$^a$, Michael D. Barrus$^b$, John Sinkovic$^c$ and\\
Ralihe R. Villagr\'an$^d$
\\ \\
{\small $^a$ Banco de M\'exico} \\
{\small Mexico City, Mexico}\\
{\small {\tt carlos.alfaro@banxico.org.mx}} \\
{\small $^b$Department of Mathematics} \\
{\small University of Rhode Island}\\
{\small Kingston, RI 02881, USA}\\
{\small {\tt barrus@uri.edu}}\\
{\small $^c$Department of Mathematics} \\
{\small Brigham Young University - Idaho}\\
{\small Rexburg, ID 83460, USA}\\
{\small {\tt sinkovicj@byui.edu}}\\
{\small $^d$Departamento de Matem\'aticas}\\ {\small Centro de Investigaci\'on y de Estudios Avanzados del IPN}\\
{\small Apartado Postal 14-740, 07000 Mexico City, Mexico} \\
{\small {\tt
rvillagran@math.cinvestav.mx}}\\
}
\date{}


\maketitle

\begin{abstract}
    We give a characterization of the graphs with at most three trivial characteristic ideals.
    This implies the complete characterization of the regular graphs whose critical groups have at most three invariant factors equal to 1 and the characterization of the graphs whose Smith groups have at most 3 invariant factors equal to 1.
    We also give an alternative and simpler way to obtain the characterization of the graphs whose Smith groups have at most 3 invariant factors equal to 1, and a list of minimal forbidden graphs for the family of graphs with Smith group having at most 4 invariant factors equal to 1.
\end{abstract}


\section{Introduction}

By considering an $m\times n$ matrix $M$ with integer entries as a linear map $M:\mathbb{Z}^n\rightarrow \mathbb{Z}^m$, the {\it cokernel} of $M$
is the quotient module $\mathbb{Z}^{m}/{\rm Im}\, M$.
This finitely generated Abelian group becomes a graph invariant when we take the matrix $M$ to be a matrix associated with the graph, say, the adjacency or Laplacian matrix.
The cokernel of the adjacency matrix $A(G)$ is known as the {\it Smith group} of $G$ and is denoted $S(G)$, and the torsion part of the cokernel of the Laplacian matrix $L(G)$ is known as the {\it critical group} $K(G)$ of $G$.

Smith groups were introduced in \cite{rushanan}.
Recently, the computation of the Smith group for several families of graphs has attracted attention, see \cite{BK,CSX,DGW,DJ,W14}.
The critical group is especially interesting for connected graphs, since its order is equal to the number of spanning trees of the graph.
The critical group has been studied intensively over the last 30 years on several contexts: the {\it group of components} \cite{lorenzini1991,lorenzini2008}, the {\it Picard group} \cite{bhn,biggs1999}, the {\it Jacobian group} \cite{bhn,biggs1999}, the {\it sandpile group} \cite{alfaval0,cori},  {\it chip-firing game} \cite{biggs1999,merino}, or {\it Laplacian unimodular equivalence} \cite{gmw,merris}.
The book of Klivans \cite{Klivans} is an excellent reference on the theory of sandpiles and its connections to other combinatorial objects like hyperplane arrangements, parking functions, dominoes, etc.

The computation of the Smith normal form (SNF) of a matrix is a standard technique to determine its cokernel.
We might refer the reader to the Stanley's survey  \cite{stanley} on SNFs in combinatorics for more details in the topic.

One way to compute the SNF of a matrix $M$ is by means of elementary row and column operations over the integers.
Let $M$ and $N$ be two $n\times n$ matrices with integer entries.
We say that $M$ and $N$ are {\it equivalent}, denoted by $N\sim M$, if there exist $P,Q\in GL_n(\mathbb{Z})$ such that $N=PMQ$.
That is, $M$ can be transformed to $N$ by applying elementary row and column operations which are invertible over the ring of integers:
\begin{enumerate}
  \item Swapping any two rows or any two columns.
  \item Adding integer multiples of one row/column to another row/column.
  \item Multiplying any row/column by $\pm 1$.
\end{enumerate}
Moreover, if $N\sim M$, then $coker(M)=\mathbb{Z}^n/{\rm Im} M\cong\mathbb{Z}^n/{\rm Im} N=coker(N)$.
Therefore, as the fundamental theorem of finitely generated Abelian groups states, the cokernel of $M$ can be described as:
$coker(M)\cong \mathbb{Z}_{d_1} \oplus \mathbb{Z}_{d_2} \oplus \cdots \oplus\mathbb{Z}_{d_{r}} \oplus \mathbb{Z}^{m-r}$,
where $d_1, d_2, \dots, d_{r}$ are positive integers with $d_i \mid d_j$ for all $i\leq j$.
These integers are called {\it invariant factors} of $M$.
Let $\phi(M)$ denote the number of invariant factors of $M$ equal to 1.

The computation of the invariant factors of the Laplacian matrix is an important technique used in the understanding of the critical group.
For instance, several researchers have addressed the question of how often the critical group is cyclic, that is, how often $\phi(L(G))$ is equal to $n-2$ or $n-1$?
In \cite{lorenzini2008} and \cite{wagner} D. Lorenzini and D. Wagner, based on numerical data, suggest we could expect to find a substantial proportion of graphs having a cyclic critical group.
Based on this, D. Wagner conjectured \cite{wagner} that almost every connected simple graph has a cyclic critical group.
A recent study \cite{wood} concluded that the probability that the critical group of a random graph is cyclic is asymptotically at most
\[
	\zeta(3)^{-1}\zeta(5)^{-1}\zeta(7)^{-1}\zeta(9)^{-1}\zeta(11)^{-1}\cdots \approx 0.7935212,
\]
where $\zeta$ is the Riemann zeta function; differing from Wagner's conjecture. 
Besides, it is interesting \cite{ck} that for any given connected simple graph, there is an homeomorphic graph with cyclic critical group.
The reader interested on this topic may consult \cite{cklpw,lorenzini2008,wood} for more questions and results.

On the other hand, the characterization of the family $\mathcal{K}_k$ of simple connected graphs having critical group with $k$ invariant factors equal to $1$ has been of great interest.
Probably, it was initially posed by R. Cori\footnote{Personal communication with C. Merino}.
However, the first result appeared  when D. Lorenzini noticed in \cite{lorenzini1991}, and independently A. Vince in \cite{vince}, that the graphs in $\mathcal K_1$ 
consist only of complete graphs.
After, C. Merino in \cite{merino} posed interest on the characterization of $\mathcal K_2$ and $\mathcal K_3$.
  In this sense, some advances have been done.
  For instance, in \cite{pan} it was characterized the graphs in $\mathcal K_2$ whose third invariant factor is equal to $n$, $n-1$, $n-2$, or $n-3$.
  In \cite{chan} the characterizations of the graphs in $\mathcal K_2$ with a cut vertex and number of independent cycles equal to $n-2$ are given.
  
  Later, a complete characterization of $\mathcal K_2$ was obtained in \cite{alfaval}.
  On the other hand, the characterization of the graphs in $\mathcal K_3$ seems to be a hard open problem \cite{alfaval1}.
  For digraphs case, the characterization of digraphs with at most 1 invariant factor equal to 1 was completely obtained in \cite{AVV}.
  These characterizations were obtained by using the {\it critical ideals} of a graph $G$, that are determinantal ideals, defined in \cite{corrval}, of the matrix $\diag(x_1,\dots,x_n)-A(G)$, where $x_1,\dots, x_n$ are indeterminates. 
  These ideals turned out \cite{alflin} to be related with other parameters like the {\it minimum rank} and the {\it zero-forcing number}.
  Similar ideals for the distance and distance Laplacian matrices were introduced in \cite{alfaro2} with the name of {\it distance ideals}.
 Therefore, for example, the family of graphs with 2 trivial distance ideals contains the family of graphs whose distance matrix has at most 2 invariant factors equal to 1.
It is interesting that there is an infinite number of minimal forbidden graphs for the graphs with two trivial distance ideals, see \cite{alfaro1}.

 In the context of the Smith groups of graphs, it would be also interesting to characterize graphs having Smith group with at most $k$ invariant factors equal to 1.
For this we introduce further notation, let $\mathcal{S}_{\leq k}$ denote the family of simple connected graphs whose adjacency matrix has at most $k$ invariant factors equal to 1, that is, $\phi(A(G))\leq k$.
The characterization of the $\mathcal S_{\leq 1}$ and $\mathcal S_{\leq 2}$ can be derived from \cite{alfaval}, and the characterization of the digraphs with $\phi(A(G))\leq 1$ was obtained in \cite{AVV}.
However, nothing is known on the structure of $\mathcal{S}_{\leq k}$, for $k\geq3$.

The manuscript is organized as follows.
In Section~\ref{section:charideals}, we introduce the concept of characteristic ideals which are determinantal ideals defined in \cite{corrval} as a generalization of the critical group and the characteristic polynomial.
Also, we give the characterization of the graphs with one and two trivial characteristic ideals, and by product the characterization of the regular graphs in $\mathcal{K}_{\leq 1}$ and $\mathcal{K}_{\leq 2}$.
We give, in Section~\ref{sec:charreg}, the characterization of graphs with 3 trivial characteristic ideals, consequently, this is used to give a complete characterization of regular graphs in $\mathcal K_{\leq3}$.
The characterization of $\mathcal{S}_{\leq 1}$, $\mathcal{S}_{\leq 2}$, and $\mathcal{S}_{\leq 3}$ can be derived from the obtained results, however, in Section~\ref{sec:smithgroup}, we give an alternative and simpler way to characterize these graph families.
We also give a list of 43 forbidden graphs for $\mathcal{S}_{\leq 4}$.

\section{Characteristic ideals of graphs}\label{section:charideals}

Consider a $n\times n$ matrix $M$ whose entries are in the polynomial ring $\mathbb{Z}[X]$ with $X=\{x_1, \dots, x_m\}$.
For $k\in [n]:=\{1, \dots, n\}$, let $\mathcal{I}=\{r_j\}_{j=1}^k$ and $\mathcal{J}=\{c_j\}_{j=1}^k$ be two sequences such that
$1\leq r_1 < r_2 < \cdots < r_k \leq n$  and $1\leq c_1 < c_2 < \cdots < c_k \leq n$.
Let $M[\mathcal{I;J}]$ denote the submatrix of a matrix $M$ induced by the rows with indices in $\mathcal{I}$ and columns with indices in $\mathcal{J}$.
Recall the determinant of $M[\mathcal{I;J}]$ is called $k$-{\it minor} of $M$.
The set of all $k$-minors of $M$ is denoted by $\minors_k(M)$.
The $k$-{\it th} {\it determinantal ideal} $I_k(M)$ of a matrix $M$ is the ideal generated by all $k$-minors of $M$.
Some properties of determinantal ideals of graphs can be found in \cite{akm}.
For example, determinantal ideals of $M$ satisfy
\begin{equation}\label{eqn:idealschain}
\langle 1\rangle \supseteq I_1(M) \supseteq \cdots \supseteq I_n(M) \supseteq \langle 0\rangle.
\end{equation}


An ideal is said to be {\it trivial}  or {\it unit} if it is equal to $\langle1\rangle$, that is, the ideal is equal to $\mathbb{Z}[X]$.

\begin{definition}
The $k$-{th} characteristic ideal $A_k(G,t)$ of a graph $G$ is the $k$-th determinantal ideal of the matrix $tI_n-A(G)$, that is, the ideal $\langle {\rm minors}_k(tI_n-A(G))\rangle\subseteq \mathbb{Z}[t]$.
The algebraic co-rank $\gamma_A(G)$ of a graph $G$ is the maximum integer $k$ such that $A_k(G,t)$ is trivial.
\end{definition}

\begin{figure}[h!]
\begin{center}
    \begin{tikzpicture}[scale=1,thick]
    \tikzstyle{every node}=[minimum width=0pt, inner sep=2pt, circle]
        \draw (-0.8,0) node[draw] (0) {};
        \draw (0.8,0) node[draw] (1) {};
        \draw (0,1) node[draw] (2) {};
        \draw (0,-1) node[draw] (3) {};
        \draw  (0) -- (1) -- (2) -- (0) -- (3) -- (1);
    \end{tikzpicture}
\end{center}
\caption{{\sf diamond} graph}
\label{figure:diamond}
\end{figure}
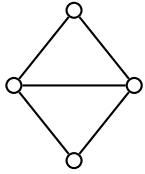

\begin{example}
Let $G$ be the {\sf diamond} graph with vertex set $V=\{v_1,v_2,v_3,v_4\}$ such that each pair of vertices are adjacent, except for $v_1$ and $v_3$, see Figure~\ref{figure:diamond}.
Then
\[
tI_4-A(G)=
\begin{bmatrix}
 t & -1 &  0 & -1\\
-1 &  t & -1 & -1\\
 0 & -1 &  t & -1\\
-1 & -1 & -1 &  t
\end{bmatrix}
\]
Since $\pm1$ is in $\minors_1(tI_4-A(G))$ and in $\minors_2(tI_4-A(G))$, then $A_1(G,t)$ and $A_2(G,t)$ are trivial.
The different 3-minors of $tI_4-A(G)$ are:
\[
t^3 - 2t, -t^2 - 2t, t^2 + t, t^3 - 3t - 2, -2t - 2.
\]
Note that $t= -(-t^2-2t) - (t^2+t)$, then $t\in A_3(G,t)$, and similarly $2\in A_2(G,t)$.
Since all the 3-minors are a linear combination of $t$ and $2$, then $A_3(G,t)=\langle 2,t\rangle$.
It is interesting to note that if $A_3(G,t)$ would be defined on $\mathbb{R}[t]$ instead, then $A_3(G,t)$ would be trivial.
Finally, $A_4(G,t) = \langle \det( tI_4-A(G) ) \rangle = \langle t^4 - 5t^2 - 4t\rangle$.
\end{example}

Computing Gr\"obner basis of a characteristic ideal are an useful computational tool to find a minimal generating set.
They can be computed in SAGE with the following code.

\begin{lstlisting}
# G is a graph
def CharIdeals(G):
    n = G.order()
    R = macaulay2.ring("ZZ",'[t,x]').to_sage()
    R.inject_variables(verbose=False);
    L = diagonal_matrix([t for i in xrange(n)]) - G.adjacency_matrix()
    Gamma = 0
    for i in range(n+1):
        M = L.minors(i)
        I = R.ideal(M).groebner_basis()
        print("Grobner basis of char ideal of size " + str(i))
        print(str(I))
        if I[0] == 1:
            Gamma = i
    print("gamma_A = " + str(Gamma))
\end{lstlisting}

\begin{example}
Thus, the Gr\"obner basis of the characteristic ideals and the algebraic co-rank of the {\sf diamond} graph can be computed with the following SAGE code:
\begin{lstlisting}
CharIdeals(Graph("C^"))
\end{lstlisting}

\end{example}

The connection of the characteristic ideals with the cokernel of the adjacency and Laplacian matrices is that the invariant factors can be recovered by evaluating the characteristic ideals.
This rely on the Theorem of elementary divisors, whose for details can be foun in \cite[Theorem 3.9]{jacobson}.

\begin{theorem}[Theorem of elementary divisors]\label{theo:elemdivisors}
Let $M$ a integer matrix of rank $r$ with $d_1,\dots,d_r$ its invariant factors.
For $k\geq1$, let $\Delta_k$ be the $\gcd$ of the $k$-minors of $M$, and $\Delta_0=1$.
Then
\[ d_k=\frac{\Delta_k}{\Delta_{k-1}}.
\]
\end{theorem}

\begin{proposition}\label{prop:characterisitidealsinvariantfactors}\cite[Corollary 15]{akm}
For $k\in [n]$,
\[
A_k(G,0)=\left< \prod_{j=1}^{k} d_j(A(G)) \right> =\left< \Delta_k(A(G))\right>,
\]
and if $G$ is r-regular, then
\[
A_k(G,r)=\left< \prod_{j=1}^{k} d_j(L(G)) \right> =\left< \Delta_k(L(G))\right>,
\]
where $\Delta_k(M)$ is the {\it greatest common divisor} of the $k$-minors of matrix $M$, and if $d_1(M)\mid\cdots\mid d_{r}(M)$ are the invariant factors in the Smith normal form of $M$, then $d_{k}(M)=\frac{\Delta_k(M)}{\Delta_{k-1}(M)}$ with $\Delta_{0}(M)=1$.
\end{proposition}

\begin{example}
Continuing with the {\sf diamond} graph.
By evaluating, $t$ at 0 in each characteristic ideal, we obtain that its SNF of the adjacency matrix is $\diag(1,1,2,0)$, meanwhile, for this case the SNF of its Laplacian matrix cannot be obtained since the {\sf diamond} graph is not regular.
\end{example}

As consequence, if $A_k(G,t)$ is trivial, then the $k$-{\it th} invariant factor $d_k(A(G))=1$, and thus, $\gamma_A(G)\leq\phi(A(G))$.
For the Laplacian matrix, we have the same when $G$ is regular, that is, if $G$ is regular, then $\gamma_A(G)\leq\phi(L(G))$.
Then, the graphs in $\mathcal S_{\leq k}$ and the regular graphs in $\mathcal K_{\leq k}$ are contained in the family $\mathcal C_{\leq k}$ of graphs with at most $k$ trivial characteristic ideals.
By characterizing the graphs in $\mathcal C_{\leq k}$, we can use the containment to give a characterization of the regular graphs in $\mathcal K_{\leq k}$.
Analogous ideas can be used to characterize $\mathcal S_{\leq k}$, however simpler ideas can be applied to obtain them, we will explore them in Section~\ref{sec:smithgroup}. 

One advantage of characteristic ideals over critical group is that characteristic ideals are monotone on induced subgraphs.

\begin{lemma}\label{lemma:CharofInducedsubgraphs}
    If $H$ is an induced subgraph of $G$, then $A_k(H,t)\subseteq A_k(G,t)$.
\end{lemma}
\begin{proof}
It follows since any $k$-minor of $tI_n-A(H)$ is also a $k$-minors of $tI_n-A(G)$.
Therefore $\minors_k(tI_n-A(H))\subseteq \minors_k(tI_n-A(H))$.
\end{proof}

A similar result is not always true for the critical group, in fact, there are examples of graphs having different critical group, for instance, $K(K_4)\ntrianglelefteq K(K_5)$.
This is because, in general, it is not  true that if $H$ is an induced subgraph of $G$, then $L(H)$ is a submatrix of $L(G)$.

A graph $G$ is \textit{forbidden} for $\mathcal{C}_{\leq k}$ if $\gamma_A(G)\geq k+1$.
Thus, we can look for the minimal forbidden graphs to characterize the family $\mathcal C_{\leq k}$.

\begin{lemma}\label{lemma:pathisforb}
The path $P_k$ with $k$ vertices is forbidden for $\mathcal C_{\leq k-2}$.
\end{lemma}

The following theorem give us the characterization of $\mathcal C_{\leq 1}$ and since the graphs in $\mathcal K_{\leq 1}$ are regular, we have $\mathcal C_{\leq 1}=\mathcal K_{\leq 1}$.
Its proof is similar to Theorem 3.3 and Corollary 3.4 of \cite{alfaval}.

\begin{theorem}
  Let $G$ be connected simple graph.
  Then the following statements are equivalent.
  \begin{enumerate}
    \item $G\in \mathcal C_{\leq 1}$,
    \item $G\in \mathcal K_{\leq 1}$
    \item $G$ is $P_3$-free
    \item $G$ is a complete graph
  \end{enumerate}
\end{theorem}

Now, before to give the characterizations of the graphs in $\mathcal C_{\leq 2}$, we give an explicit formula of the characteristic ideals of complete graphs and complete multipartite graphs, and few structural results needed for the characterization.

\begin{lemma}\cite[Proposition 3.15 \& Theorem 3.16]{corrval}\label{lemma:CharCompletegraphs}
Let $G$ be a complete graph with $n$ vertices.
Then
\[
A_j(G,t)=
\begin{cases}
\left\langle (t+1)^{j-1}\right\rangle & 1\leq j\leq n-1,\\
\left\langle (t-n+1)(t+1)^{n-1}\right\rangle & j= n.
\end{cases}
\]
\end{lemma}

\begin{lemma}\cite[Theorem 3.2]{yibo}\label{lemma:CharCompleteMultipartite}
  Let $G$ be a complete multipartite graph with $m\geq 2$ parts of size $r_1,\dots,r_m\geq2$.
  Let $n=\sum r_i$.
  Then
  \[
  A_j(G,t)=
\begin{cases}
\langle1\rangle & j\leq m-1,\\
\left\langle (m-1)t^{j-m},t^{j-m+1} \right\rangle & m\leq j\leq n-m,\\
\left\langle t^{j-m+1}\prod_{a=1}^{m-n+j-1}(t+r_{i_a}), P \right\rangle & n-m < j < n\,\\
\left\langle \sum_{a=0}^m e_a(r_1,\dots,r_m)t^{n-a} \right\rangle & j = n,\\
\end{cases}
\]
where $P$ is equal to 
\scalebox{0.7}{
$\displaystyle\left\{ \sum_{a=0}^{m-k}(k-1+a)e_a(r_{i_1},\dots, r_{i_{m-k}})t^{j-k-a}\; :\; k=n-j\geq1 \text{ and } 1\leq i_1<\cdots< i_{m-k}\leq m \right\}
$
},
and $e_a(s_1,\dots,s_l)$ is the elementary symmetric polynomial of degree $a$ in $l$ variables, {\it i.e.},
\[
e_a(s_1,\dots,s_l)=\sum_{1\leq s_{i_1}<\cdots< s_{i_a}\leq l} s_{i_1}\cdots s_{i_a}.
\]
\end{lemma}

\begin{lemma}\cite[Theorem 1]{O88}\label{lemma:pawfree}
  Let $G$ be a {\sf paw}-free connected graph.
  Then $G$ is either $K_3$-free or complete multipartite graph.
\end{lemma}

\begin{lemma}\cite[Proposition 1]{BvL}\label{lemma:P4pawfree}
  Let $G$ be a $\{P_4,K_3\}$-free connected graph, then $G$ is a complete bipartite graph.
\end{lemma}

\begin{theorem}\label{theo:familyK2}
  Let $G$ be connected simple graph.
  Then the following statements are equivalent:
  \begin{enumerate}
    \item $G\in {\mathcal C}_{\leq 2}$,
    \item $G$ is $\{P_4,{\sf paw},K_5-e\}$-free, 
    \item $G$ is complete graph or $G$ is an induced subgraph of a complete tripartite graph.
  \end{enumerate}
\end{theorem}
\begin{proof}
$(1)\implies (2)$ By Lemma~\ref{lemma:pathisforb}, $P_4$ is forbidden for $\mathcal C_{\leq 2}$.
Now considering
\[
M=t{\sf I}_4 -A({\sf paw})=
\begin{bmatrix}
 t & -1 &  0 &  0\\
-1 &  t & -1 & -1\\
 0 & -1 &  t & -1\\
 0 & -1 & -1 &  t
\end{bmatrix},
\]
we can obtain that $A_1({\sf paw},t)$ and $A_2({\sf paw},t)$ are trivial since there are appropriate minors of $M$ equal to 1.
Let $$p(t)=\det(M[\{1,2,3\};\{1,2,4\}])=-t^2-t+1$$ and $$q(t)=\det(M[\{1,2,3\};\{1,3,4\}])=t^2+t.$$
Since $1=p(t)+q(t)\in A_3({\sf paw},t)$, then $A_3({\sf paw},t)$ is trivial.
Thus $\sf paw$ is forbidden for $\mathcal C_{\leq 2}$.
Now, let
\[
M=t{\sf I_5}-A(K_5-e)=
\begin{bmatrix}
 t &  0 & -1 & -1 & -1\\
 0 &  t & -1 & -1 & -1\\
-1 & -1 &  t & -1 & -1\\
-1 & -1 & -1 &  t & -1\\
-1 & -1 & -1 & -1 &  t
\end{bmatrix}.
\]
And, let $$p(t)=\det(M[\{1,2,3\};\{1,2,4\}])=-t^2-2t$$ and $$q(t)=\det(M[\{2,3,4\};\{2,3,5\}])=-t^2-2t-1.$$
Since $1=p(t)-q(t)$, then $A_3(K_5-e,t)$ is trivial.
From which follows that $K_5-e$ is forbidden for ${\mathcal C}_{\leq2}$.

$(2)\implies (3)$ By Lemma~\ref{lemma:pawfree}, a $\sf paw$-free graph is either $K_3$ or a complete multipartite graph.
In the first case, considering that $G$ is also $P_4$-free, then by Lemma~\ref{lemma:P4pawfree}, $G$ is a bipartite graph.
On the other hand, let $G$ be a complete multipartite graph with more than 3 partite sets.
Since $G$ is $\{K_5-e\}$-free, then each partite set has at most one vertex, that is, $G$ is a complete graph.

$(3)\implies (1)$ Lemma~\ref{lemma:CharCompletegraphs} states complete graphs have at most one trivial characteristic ideal.
Now let $G$ be a complete tripartite graph with each part of size at least 2.
By Lemma~\ref{lemma:CharCompleteMultipartite}, we have the third characteristic ideal is not trivial.
Thus by Lemma~\ref{lemma:CharofInducedsubgraphs}, if $H$ is an induced subgraph of $G$, then $H$ has at most 2 trivial characteristic ideals.
\end{proof}

The characterization of the regular graphs whose critical group have 2 invariant factors equal to 1 follows by evaluating the third characteristic ideal of these graphs at $t$ equal the degree of any vertex.

\begin{corollary}
Let $G$ be a connected simple regular graph.
Then $G\in \mathcal K_{\leq 2}$ if and only if $G$ is either a complete graph $K_r$, a regular complete bipartite graph $K_{r,r}$ or a regular complete tripartite graph $K_{r,r,r}$.
\end{corollary}
\begin{proof}
Since $G$ is regular and $G\in \mathcal{C}_{\leq 2}$, then $G$ is either a complete graph $K_r$, a regular complete bipartite graph $K_{r,r}$ or a regular complete tripartite graph $K_{r,r,r}$.
On the other hand, let $G$ be any of these graphs.
By Lemmas~\ref{lemma:CharCompletegraphs} and \ref{lemma:CharCompleteMultipartite}, the third characteristic ideal of $G$ is
\[
A_3(G,t)=
\begin{cases}
\langle (t+1)^2(t-2) \rangle & K_{r+1} \text{ with } r= 2,\\
\langle (t+1)^2 \rangle & K_{r+1} \text{ with } r\geq 3,\\
\langle t^2,2t \rangle & K_{r,r} \text{ with } r=2,\\
\langle t \rangle & K_{r,r} \text{ with } r\geq 3,\\
\langle 2,t \rangle & K_{r,r,r} \text{ with } r\geq 2.\\
\end{cases}
\]
By evaluating $A_2(G,t)$ and $A_3(G,t)$ at $t$ equal the degree of any vertex of $G$, we obtain that the third invariant factor of $G$ is different than 1.
\end{proof}

A characterization of the graphs with Smith groups having 2 invariant factors equal to 1 can also be obtained by evaluating the third characteristic ideal of a complete graph or and induced subgraph of a tripartite graph at $t=0$, however, we will use simpler ideas in Section~\ref{sec:smithgroup}.

\section{Regular graphs with at most 3 trivial characteristic ideals}\label{sec:charreg}

In this section we will characterize the graphs with at most 3 trivial characteristic ideals.
As consequence, we will obtain a complete characterization of the regular graphs in $\mathcal{K}_{\leq 3}$.

\begin{figure}[h!]
\begin{center}
		\begin{tikzpicture}[scale=1,thick]
		\tikzstyle{every node}=[minimum width=0pt, inner sep=2pt, circle]
			\draw (-0.92,1.16) node[draw] (0) {};
			\draw (-0.77,0.63) node[draw] (1) {};
			\draw (0.52,1.43) node[draw] (2) {};
			\draw (0.66,0.89) node[draw] (3) {};
			\draw (-2.52,1.4) node[draw] (4) {};
			\draw (-2.59,0.83) node[draw] (5) {};
			\draw (-0.94,-0.96) node[draw] (6) {};
			\draw  (0) edge (4);
			\draw  (0) edge (2);
			\draw  (1) edge (2);
			\draw  (1) edge (5);
			\draw  (1) edge (6);
			\draw  (2) edge (3);
			\draw  (4) edge (5);
      \draw (5) -- (0) -- (3) -- (1) -- (4);
      \draw (0) -- (6);
		\end{tikzpicture}
	\end{center}
\caption{$S_4^{\mathbf{r}}$ with $\mathbf{r}=(2,1,-2,-2)$}
\label{fig:thickstar}
\end{figure}
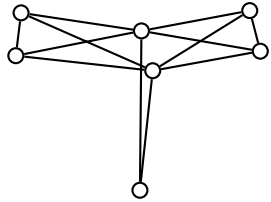

Given a graph $G=(V,E)$ and a vector ${\bf d}\in {\mathbb Z}^V$, the graph $G^{\bf d}$ is constructed as follows.
For each vertex $u\in V$, associate a new vertex set $V_u$, where $V_u$ is a clique of cardinality $-{\bf d}_u$ when ${\bf d}_u$ is negative, and $V_u$ is a stable set of cardinality ${\bf d}_u$ if ${\bf d}_u$ when positive.
Each vertex in $V_u$ is adjacent with each vertex in $V_v$ if and only if $u$ and $v$ are adjacent in $G$.
Then the graph $G$ is called the {\it underlying graph} of $G^{\bf d}$.
For instance, let $S_n$ denote the star graph with $n$ vertices; with one apex vertex and $n-1$ leaves.
In Figure~\ref{fig:thickstar} there is a drawing of $S_4^{\mathbf{r}}$ with $\mathbf{r}=(2,1,-2,-2)$, where the first entry of $\bf r$ is associated with the apex vertex.

\begin{figure}
\begin{tabular}{c@{\extracolsep{1cm}}c@{\extracolsep{1cm}}c@{\extracolsep{1cm}}c}
    \begin{tikzpicture}[scale=.7,thick]
    \tikzstyle{every node}=[minimum width=0pt, inner sep=2pt, circle]
        \draw (0:1) node[draw] (0) { };
        \draw (72:1) node[draw] (3) { };
        \draw (144:1) node[draw] (2) { };
        \draw (216:1) node[draw] (1) { };
        \draw (288:1) node[draw] (4) { };
        \draw  (0) edge (3);
        \draw  (0) edge (4);
        \draw  (1) edge (4);
        \draw  (2) edge (4);
    \end{tikzpicture}
&
    \begin{tikzpicture}[scale=.7,thick]
    \tikzstyle{every node}=[minimum width=0pt, inner sep=2pt, circle]
        \draw (0:1) node[draw] (3) { };
        \draw (72:1) node[draw] (0) { };
        \draw (144:1) node[draw] (2) { };
        \draw (216:1) node[draw] (4) { };
        \draw (288:1) node[draw] (1) { };
        \draw  (0) edge (3);
        \draw  (0) edge (4);
        \draw  (1) edge (3);
        \draw  (1) edge (4);
        \draw  (2) edge (4);
    \end{tikzpicture}
&
    \begin{tikzpicture}[scale=.7,thick]
    \tikzstyle{every node}=[minimum width=0pt, inner sep=2pt, circle]
        \draw (0:1) node[draw] (0) { };
        \draw (72:1) node[draw] (4) { };
        \draw (144:1) node[draw] (2) { };
        \draw (216:1) node[draw] (3) { };
        \draw (288:1) node[draw] (1) { };
        \draw  (0) edge (3);
        \draw  (0) edge (4);
        \draw  (1) edge (3);
        \draw  (2) edge (4);
        \draw  (3) edge (4);
    \end{tikzpicture}
&
    \begin{tikzpicture}[scale=.7,thick]
    \tikzstyle{every node}=[minimum width=0pt, inner sep=2pt, circle]
        \draw (0:1) node[draw] (3) { };
        \draw (72:1) node[draw] (1) { };
        \draw (144:1) node[draw] (2) { };
        \draw (216:1) node[draw] (4) { };
        \draw (288:1) node[draw] (0) { };
        \draw  (0) edge (3);
        \draw  (0) edge (4);
        \draw  (1) edge (3);
        \draw  (1) edge (4);
        \draw  (2) edge (4);
        \draw  (3) edge (4);
    \end{tikzpicture}
\\
$\Gaa$ & $\Gab$ & $\Gac$ & $\Gad$
\\
    \begin{tikzpicture}[scale=.7,thick]
    \tikzstyle{every node}=[minimum width=0pt, inner sep=2pt, circle]
        \draw (0:1) node[draw] (0) { };
        \draw (72:1) node[draw] (2) { };
        \draw (144:1) node[draw] (3) { };
        \draw (216:1) node[draw] (1) { };
        \draw (288:1) node[draw] (4) { };
        \draw  (0) edge (2);
        \draw  (0) edge (4);
        \draw  (1) edge (3);
        \draw  (1) edge (4);
    \end{tikzpicture}
&
    \begin{tikzpicture}[scale=.7,thick]
    \tikzstyle{every node}=[minimum width=0pt, inner sep=2pt, circle]
        \draw (0:1) node[draw] (0) { };
        \draw (72:1) node[draw] (2) { };
        \draw (144:1) node[draw] (3) { };
        \draw (216:1) node[draw] (1) { };
        \draw (288:1) node[draw] (4) { };
        \draw  (0) edge (2);
        \draw  (0) edge (4);
        \draw  (1) edge (3);
        \draw  (1) edge (4);
        \draw  (2) edge (4);
    \end{tikzpicture}
&
    \begin{tikzpicture}[scale=.7,thick]
    \tikzstyle{every node}=[minimum width=0pt, inner sep=2pt, circle]
        \draw (0:1) node[draw] (2) { };
        \draw (72:1) node[draw] (0) { };
        \draw (144:1) node[draw] (3) { };
        \draw (216:1) node[draw] (1) { };
        \draw (288:1) node[draw] (4) { };
        \draw  (0) edge (2);
        \draw  (0) edge (3);
        \draw  (0) edge (4);
        \draw  (1) edge (3);
        \draw  (1) edge (4);
        \draw  (2) edge (4);
        \draw  (3) edge (4);
    \end{tikzpicture}
&
    \begin{tikzpicture}[scale=.7,thick]
    \tikzstyle{every node}=[minimum width=0pt, inner sep=2pt, circle]
        \draw (0:1) node[draw] (0) { };
        \draw (72:1) node[draw] (3) { };
        \draw (144:1) node[draw] (2) { };
        \draw (216:1) node[draw] (4) { };
        \draw (288:1) node[draw] (1) { };
        \draw  (0) edge (2);
        \draw  (0) edge (3);
        \draw  (0) edge (4);
        \draw  (1) edge (4);
        \draw  (2) edge (3);
        \draw  (2) edge (4);
    \end{tikzpicture}
\\
$\Gae$ & $\Gaf$ & $\Gag$ & $\Gah$
\\
    \begin{tikzpicture}[scale=.7,thick]
    \tikzstyle{every node}=[minimum width=0pt, inner sep=2pt, circle]
        \draw (0:1) node[draw] (0) { };
        \draw (60:1) node[draw] (1) { };
        \draw (120:1) node[draw] (2) { };
        \draw (180:1) node[draw] (3) { };
        \draw (240:1) node[draw] (5) { };
        \draw (300:1) node[draw] (4) { };
        \draw  (0) edge (4);
        \draw  (0) edge (5);
        \draw  (1) edge (5);
        \draw  (2) edge (5);
        \draw  (3) edge (5);
        \draw  (4) edge (5);
    \end{tikzpicture}
&
    \begin{tikzpicture}[scale=.7,thick]
    \tikzstyle{every node}=[minimum width=0pt, inner sep=2pt, circle]
        \draw (0:1) node[draw] (5) { };
        \draw (60:1) node[draw] (1) { };
        \draw (120:1) node[draw] (2) { };
        \draw (180:1) node[draw] (4) { };
        \draw (240:1) node[draw] (3) { };
        \draw (300:1) node[draw] (0) { };
        \draw  (0) edge (3);
        \draw  (0) edge (4);
        \draw  (0) edge (5);
        \draw  (1) edge (4);
        \draw  (1) edge (5);
        \draw  (2) edge (4);
        \draw  (2) edge (5);
        \draw  (3) edge (4);
        \draw  (3) edge (5);
    \end{tikzpicture}
&
    \begin{tikzpicture}[scale=.7,thick]
    \tikzstyle{every node}=[minimum width=0pt, inner sep=2pt, circle]
        \draw (0:1) node[draw] (0) { };
        \draw (60:1) node[draw] (4) { };
        \draw (120:1) node[draw] (2) { };
        \draw (180:1) node[draw] (3) { };
        \draw (240:1) node[draw] (1) { };
        \draw (300:1) node[draw] (5) { };
        \draw  (0) edge (3);
        \draw  (0) edge (4);
        \draw  (0) edge (5);
        \draw  (1) edge (3);
        \draw  (1) edge (4);
        \draw  (1) edge (5);
        \draw  (2) edge (3);
        \draw  (2) edge (4);
        \draw  (2) edge (5);
        \draw  (3) edge (5);
    \end{tikzpicture}
&
    \begin{tikzpicture}[scale=.7,thick]
    \tikzstyle{every node}=[minimum width=0pt, inner sep=2pt, circle]
        \draw (0:1) node[draw] (3) { };
        \draw (60:1) node[draw] (1) { };
        \draw (120:1) node[draw] (2) { };
        \draw (180:1) node[draw] (0) { };
        \draw (240:1) node[draw] (4) { };
        \draw (300:1) node[draw] (5) { };
        \draw  (0) edge (2);
        \draw  (0) edge (3);
        \draw  (0) edge (4);
        \draw  (0) edge (5);
        \draw  (1) edge (2);
        \draw  (1) edge (3);
        \draw  (1) edge (4);
        \draw  (1) edge (5);
        \draw  (2) edge (4);
        \draw  (2) edge (5);
        \draw  (3) edge (5);
        \draw  (4) edge (5);
    \end{tikzpicture}
\\
$\Gai$ & $\Gaj$ & $\Gak$ & $\Gal$
\\
&
    \begin{tikzpicture}[scale=.7,thick]
    \tikzstyle{every node}=[minimum width=0pt, inner sep=2pt, circle]
        \draw (0:1) node[draw] (0) { };
        \draw (360/7:1) node[draw] (2) { };
        \draw (720/7:1) node[draw] (1) { };
        \draw (1080/7:1) node[draw] (3) { };
        \draw (1440/7:1) node[draw] (5) { };
        \draw (1800/7:1) node[draw] (4) { };
        \draw (2160/7:1) node[draw] (6) { };
        \draw  (0) edge (2);
        \draw  (0) edge (3);
        \draw  (0) edge (4);
        \draw  (0) edge (5);
        \draw  (0) edge (6);
        \draw  (1) edge (2);
        \draw  (1) edge (3);
        \draw  (1) edge (4);
        \draw  (1) edge (5);
        \draw  (1) edge (6);
        \draw  (2) edge (4);
        \draw  (2) edge (5);
        \draw  (2) edge (6);
        \draw  (3) edge (4);
        \draw  (3) edge (5);
        \draw  (3) edge (6);
        \draw  (4) edge (5);
        \draw  (4) edge (6);
        \draw  (5) edge (6);
    \end{tikzpicture}
&
    \begin{tikzpicture}[scale=.7,thick]
    \tikzstyle{every node}=[minimum width=0pt, inner sep=2pt, circle]
        \draw (45:1.41) node[draw] (0) { };
        \draw (135:1.41) node[draw] (1) { };
        \draw (225:1.41) node[draw] (2) { };
        \draw (315:1.41) node[draw] (3) { };
        
        \draw (0:1) node[draw] (4) { };
        \draw (90:1) node[draw] (5) { };
        \draw (180:1) node[draw] (6) { };
        \draw (270:1) node[draw] (7) { };
        \draw  (0) edge (4);
        \draw  (0) edge (5);
        \draw  (0) edge (6);
        \draw  (0) edge (7);
        \draw  (1) edge (4);
        \draw  (1) edge (5);
        \draw  (1) edge (6);
        \draw  (1) edge (7);
        \draw  (2) edge (4);
        \draw  (2) edge (5);
        \draw  (2) edge (6);
        \draw  (2) edge (7);
        \draw  (3) edge (4);
        \draw  (3) edge (5);
        \draw  (3) edge (6);
        \draw  (3) edge (7);
        \draw  (4) edge (5);
        \draw  (4) edge (6);
        \draw  (4) edge (7);
        \draw  (5) edge (6);
        \draw  (5) edge (7);
        \draw  (6) edge (7);
    \end{tikzpicture}
&
\\
 & $\Gam$ & $\Gan$ &
\\
\end{tabular}
\label{fig:forbiddenchar3}
\caption{The family of graphs $\mathcal F$.}
\end{figure}
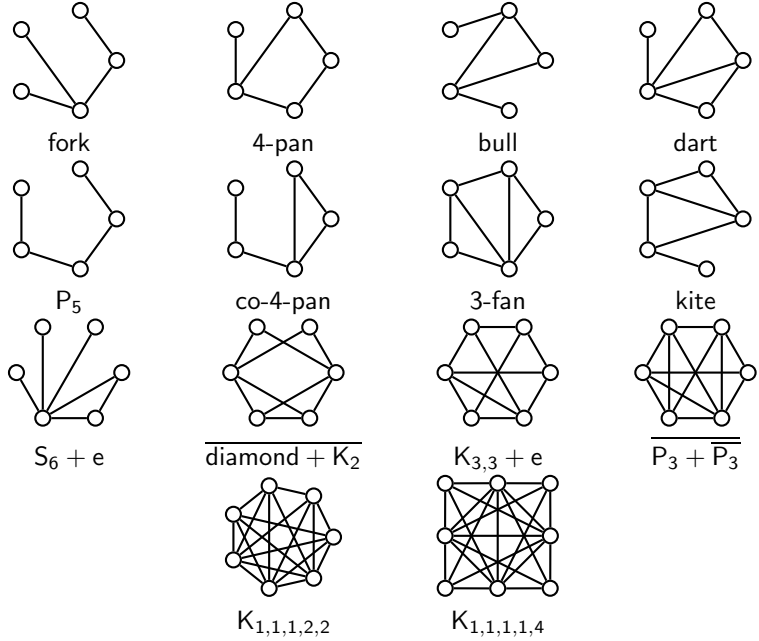

Let $\F$ denote the collection of graphs shown in Figure~\ref{fig:forbiddenchar3}.
In the following, we seek to find a structural characterization for graphs containing none of the 14 given graphs  in $\F$ as an induced subgraph.

\begin{lemma}\label{lemma:GinC_leq3isFfree}
Let $G$ be a connected graph in $\mathcal C_{\leq 3}$, then $G$ is $\F$-free.
\end{lemma}
\begin{proof}
It follows by computing the fourth characteristic ideals of the graphs in $\F$ and see that they are trivial.
Then, by Lemma~\ref{lemma:CharofInducedsubgraphs}, $G$ cannot contain any graph in $\F$ as induced subgraph.
\end{proof}

\begin{theorem}\label{teo:charC_leq3.1}
A connected graph $G$ is $\F$-free if and only if it is an induced subgraph of one of the following:
\begin{enumerate}
\item[\textup{(1)}] $C_5$,
\item[\textup{(2)}] the {\sf triangular prism} $K_3 \Box K_2$,
\item[\textup{(3)}] a complete $4$-partite graph,
\item[\textup{(4)}]  $C_4^{\mathbf{r}}$, for some $-\mathbf{r}\in \mathbb{N}^4$, or
\item[\textup{(5)}] $S_4^{\mathbf{r}}$, for some $-\mathbf{r}\in \mathbb{N}^4$.
\end{enumerate}
\end{theorem}

\begin{proof}
It is straightforward to verify that graphs of the forms specified can induce no subgraph from $\F$. Suppose henceforth that $G$ is a connected $\F$-free graph; we show that $G$ has one of the five forms described above.

Since $G$ is connected, the well known result of Seinsche \cite{seinsche} implies that $G$ either contains $\Pf$  as an induced subgraph, or $G$ is the complement of a disconnected graph and hence is a join of two graphs with nonempty vertex sets.

Suppose first that $G$ contains $\Pf$ as an induced subgraph, and let $w,x,y,z$ be the vertices, in order, of such an induced path.

Since $G$ is $\{\Gaa,\Gab,\Gac, \Gae,\Gaf, \Gag,\Gah\}$-free, we conclude that any vertex of $G$ not in $\{w,x,y,z\}$ is adjacent to either none of these four vertices, or it is adjacent to both endpoints $w,z$ and at most one of the midpoints $x,y$. Hence we may partition the vertices of $G-\{w,x,y,z\}$ into three sets:

\bigskip
\begin{tabular}{rl}
$V_{wz}$: & vertices adjacent to $w,z$ and neither of $x,y$;\\
$V_{wxz}$: & vertices adjacent to $w,z$ and $x$ but not $y$;\\
$V_{wyz}$: & vertices adjacent to $w,z$ and $y$ but not $x$;\\
$U$: & vertices adjacent to no vertex of $\{w,x,y,z\}$.
\end{tabular}

\bigskip
We illustrate these sets in Figure~\ref{figure:P4freecase}.

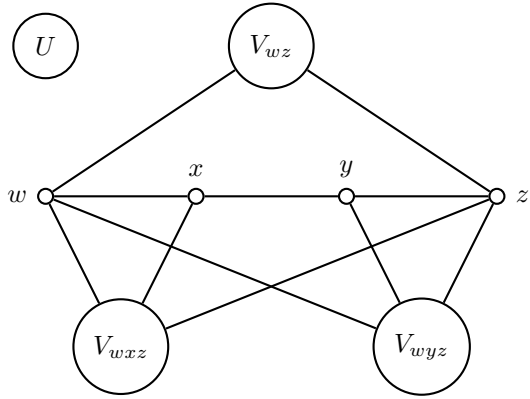
\begin{figure}[h!]
\begin{center}
\begin{tikzpicture}[scale=2,thick]
		\tikzstyle{every node}=[minimum width=0pt, inner sep=2pt, circle]
			\draw (-1.5,0) node[draw,label=left:{$w$}] (0) {};
			\draw (-0.5,0) node[draw,label=above:{$x$}] (1) {};
			\draw (0.5,0) node[draw,label=above:{$y$}] (2) {};
			\draw (1.5,0) node[draw,label=right:{$z$}] (3) {};
			\draw (0,1) node[draw,inner sep=5pt] (4) {$V_{wz}$};
			\draw (-1,-1) node[draw,inner sep=5pt] (5) {$V_{wxz}$};
			\draw (1,-1) node[draw,inner sep=5pt] (6) {$V_{wyz}$};
      \draw (-1.5,1) node[draw,inner sep=5pt] (7) {$U$};
			\draw  (0) edge (1);
			\draw  (0) edge (5);
			\draw  (0) edge (6);
			\draw  (1) edge (2);
			\draw  (2) edge (3);
			\draw  (3) edge (6);
			\draw  (3) edge (4);
			\draw  (4) edge (0);
			\draw  (5) edge (1);
			\draw  (5) edge (3);
			\draw  (6) edge (2);
		\end{tikzpicture}
\end{center}
\label{figure:P4freecase}
\caption{Diagram describing $G$.}
\end{figure}

If there is a pair of non-adjacent vertices in $V_{wz}$ then $G[\{w,y,z\}\cup V_{wz}]$ contains an induced copy of the $\Gab$. Moreover, if $|V_{wz}|\geq 2$ then $G[\{w,y,z\}\cup V_{wz}]$ contains an induced copy of the $\Gad$. Thus $|V_{wz}|\leq 1$. 

Similarly, we have$|V_{wxz}|\leq 1$ and $|V_{wyz}|\leq 1$. If $V_{wz}$ is nonempty, let us denote $V_{wz}=\{ v_{wz} \}$ and so on. 

Since the induced subgraph of $G$ having vertex set $\{x,y,z,v_{wz},v_{wxz} \} $ is not isomorphic to the $\Gab$, it must be the case that $v_{wz}$ is adjacent to $v_{wxz}$. However, then the induced subgraph on $\{x,y,z,v_{wz},v_{wxz} \} $ is isomorphic to the $\Gah$, a contradiction. Since a similar contradiction arises for vertices $v_{wz}$ and $v_{wyz}$, we conclude that if $V_{wz}$ is nonempty then both $V_{wxz}$ and $V_{wyz}$ are empty; if either $V_{wxz}$ or $V_{wyz}$ is nonempty, then $V_{wz}$ is empty.

Let $E[A,B]$ be the set of edges between two sets of vertices $A$ and $B$. If $E[V_{wxz},V_{wyz}]\neq \emptyset$, then $G[\{w,x,z,v_{wxz},v_{wyz} \}]$ contains the $\Gag$ as an induced subgraph, a contradiction, so there are no edges between $V_{wxz}$ and $V_{wyz}$.

Now note that if any vertex in $U$ has a neighbor in $V_{wz}$, then $G$ induces $P_5$, a contradiction. If $U$ has any neighbor in $V_{wxz}$ (or in $V_{wyz}$), then $G$ induces both $\Gac$ and $\Gab$. Since $G$ is connected, some vertex in $U$ would have a neighbor in $V_{wz}$ or $V_{wxz}$ or $V_{wyz}$ unless $U$ were empty, so we conclude that $U$ is empty. 

We conclude that $G$ is isomorphic to either $P_4$, $C_5$, the {\sf house} graph, or the {\sf triangular prism}. This completes the characterization of $G$ when $G$ induces $P_4$.

Suppose henceforth that $G$ is $P_4$-free. As described previously, since $G$ is a connected $P_4$-free graph, then $G$ can be written as $G=G_1 \vee G_2$, where $G_1$ and $G_2$ each have at least one vertex. Not every such graph is $\F$-free, as the graphs in Table~\ref{table:P4-freeInF} show.

\begin{table}
\begin{center}
\begin{tabular}{ll}
\hline
\hline
name & alternative name\\
\hline
$\Gad$ & $K_1 \vee (P_3+K_1)$\\
$\Gag$ & $P_4\vee K_1$\\
$\Gai$ & $K_1\vee (K_2+3K_1)$\\
$\Gaj$ & $2K_1\vee (K_2+2K_1)$\\
$\Gak$ & $3K_1\vee (K_2+K_1)$\\
$\Gal$ & $P_3\vee (K_2+K_1)=K_{1,2}\vee (K_2+K_1)$\\
$\Gam$ & $K_3\vee C_4$\\
$\Gan$ & $K_4\vee 4K_1$\\
\hline
\end{tabular}
\end{center}
\label{table:P4-freeInF}
\caption{Graphs in $\F$ that are join of two graphs.}
\end{table}

If $G$ is $K_2+K_1$-free, then, by Lemma~\ref{lemma:pawfree}, $G$ is a complete multipartite graph.
Since $G$ is $\{\Gam,\Gan\}$-free, if such a graph $G$ has five or more  partite sets, then no partite set can have four or more vertices, and at most one partite set can have two or  three vertices. Thus, if $G$ has five or more partite sets, then $G$ is isomorphic to $3K_1 \vee K_m$ or to $2K_1\vee K_m$ for some $m\geq 4$; which are included in the case (5).

If $G$ contains $K_2+K_1$, it must do so within $G_1$ or within $G_2$. Without loss of generality, suppose that $G_2$ contains $K_2+K_1$, and assume that $G_2$ cannot be written as a join of smaller graphs (if it could, we could redefine $G_1$ to include one of the vertex sets of this join). The forbidden subgraph assumptions imply that $G_1$ must be $\{P_3,3K_1\}$-free. Since $G_1$ is $P_3$-free, it is a disjoint union of cliques. 
And since $G_1$ is $3K_1$-free, there are at most two of these cliques. Hence $G_1$ has the form $K_p+K_q$, where $0 \leq p \leq q$ and $q \geq 1$ (by our assumption that the join $G=G_1\vee G_2$ was nontrivial).

If $p\geq 1$ and $q \geq 2$, then $G_1$ contains $K_2+K_1$ as an induced subgraph, and exchanging the roles of $G_1$ and $G_2$ in the arguments above imply that $G_2$ has the form $K_{p'}+K_{q'}$ for $p'\geq 1$ and $q' \geq 2$ and henceforth $G=(K_p+K_q)\vee (K_{p'}+K_{q'})$; which is included in case (4).

Next, we will consider the cases when $p=1,q=1$ and when $p=0$ in detail. First we establish some further structure for $G_2$.

Consider an induced copy of $K_2+K_1$ within $G_2$, and let $v$ be a vertex of $G_2$ not in this induced subgraph. Since $G$ is  $K_1 \vee (P_3+K_1)$-free, we may assume that $G_2$ is $\{P_4,P_3+K_1\}$-free.
And this implies that if $v$ is adjacent to one endpoint of the $K_2$-component in the $K_2+K_1$-subgraph, then it must be adjacent to the other endpoint.

Let $ab$ be the edge and let $c$ be the isolated vertex in an induced subgraph isomorphic to $K_2+K_1$. Let $X_d$ be the set of vertices in $G_2$ adjacent to none of $a,b,c$; let $X_{ab}$ be the set of vertices in $G_2$ adjacent to both $a$ and $b$ but not $c$; let $X_c$ be the set of vertices in $G_2$ adjacent to $c$ but not $a$ and $b$; and let $X_{abc}$ be the set of vertices in $G_2$ adjacent to all of $a,b,c$.

Now, if $p=q=1$, then, since $G$ is $\{2K_1\vee(K_2+2K_1)\}$-free, we may also conclude that $G_2$ is $K_2+2K_1$-free, which implies that $X_d$ is empty. 
And $X_{abc}$ is empty as well, this is because, otherwise, we would have $P_3 \vee (K_2+K_1)$ as an induced subgraph of $G$.
The vertex sets $X_{ab}$ and $X_c$ must be cliques, otherwise, $G$ would contain a $K_1 \vee (P_3+K_1)$ or a $2K_1\vee (K_2+2K_1)$, respectively.
And $E[X_{ab},X_{c}]$ is empty, since otherwise $G$ would contains $P_4$ as induced subgraph.
Therefore, $G_2$ is the disjoint union of two cliques, that is, $G=2K_1 \vee (K_r+K_s)$ with $r\geq 2$ and $s\geq 1$. Which is contained in case (4).

On the other hand, let us consider the case $p=0$ and $q\geq 1$. Then $G_1=K_q$ and $V(G_2)=\{a,b,c \}\cup X$, where $X=X_{ab}\cup X_{abc}\cup X_c\cup X_d$.

The sets $X_{ab}$ and $X_c$ are cliques, since otherwise $G_2[\{a,c \}\cup X_{ab} ]$ and $G_2[\{a,c \}\cup X_c]$ would, respectively, contain an induced copy of $(P_3+K_1)$. Also $X_d$ is a clique, since otherwise $G$ would contain an induced copy of $K_1 \vee (K_2+3K_1)$.

Furthermore $E[X_c,X_d]=\emptyset =E[X_{ab},X_d]$, otherwise $G_2[\{b,c\}\cup X_c\cup X_d]$ or $G_2[\{b,c\}\cup X_{ab}\cup X_d]$ would contain an induced copy of $P_3+K_1$, respectively. Likewise, $E[X_{ab},X_c]=\emptyset$ since otherwise $P_4$ would be an induced subgraph of $G_2[\{b,c\}\cup X_{ab}\cup X_c]$.

Moreover, $E[X_{abc},X_d]$ is of maximum size, that is, every vertex of $X_{abc}$ is adjacent to every vertex of $X_d$, since otherwise $G_2[\{b,c\}\cup X_{abc}\cup X_d]$ would contain an induced copy of $P_3+K_1$. Also $E[X_{abc},X_{ab}]$ and $E[X_{abc},X_c]$ are of maximum size because $G_2$ is $P_4$-free. 

By the argument above and our assumption that $G_2$ cannot be written as a join of smaller graphs we can conclude that $X_{abc}=\emptyset$.

Finally, if $p=0$ then $G=K_q \vee (K_r+K_s+K_t)$, where $r\geq 2$, $q,s\geq 1$ and $t\geq 0$. 
Which is included in case (5).

\end{proof}

\begin{lemma}\label{lemma:gb_C5_K3K2}
The third characteristic ideals of $C_5$ and $K_3 \Box K_2$ are trivial and
the fourth characteristic ideals of $C_5$ and $K_3 \Box K_2$ are non trivial. In fact, 
$A_4(C_5,t)=\langle t^2+t-1\rangle$ and $A_4(K_3 \Box K_2,t)=\langle t+2,5\rangle$.
\end{lemma}

\begin{obs}\label{obs: 4-th minors}
In the following, let $L_{m}=(t+1){\sf I_m}-{\sf J_m}$. Note that, for any ${\bf r}$ such that $-{\bf r}\in \mathbb{N}^4$, the 4-minors of the matrices $t{\sf I_{\bf r_1+r_2+r_3+r_4}}-A\left(C_4^{\bf r}\right)$ and $t{\sf I_{\bf r_1+r_2+r_3+r_4}}-A\left(S_4^{\bf r}\right)$ are contained in the 4-minors of the matrices
\[
t{\sf I_{16}}-A\left(C_4^{(-4,-4,-4,-4)}\right)=
\begin{bmatrix}
L_4 & -{\sf J_4} & {\sf 0}_4 & -{\sf J_4}\\
-{\sf J_4} & L_4 & -{\sf J_4} & {\sf 0}_4\\
{\sf 0}_4 & -{\sf J_4} & L_4 & -{\sf J_4}\\
-{\sf J_4} & {\sf 0}_4 & -{\sf J_4} & L_4
\end{bmatrix}
\]
and
\[
t{\sf I_{16}}-A\left(S_4^{(-4,-4,-4,-4)}\right)=
\begin{bmatrix}
L_4 & -{\sf J_4} & -{\sf J_4} & -{\sf J_4}\\
-{\sf J_4} & L_4 & {\sf 0}_4 & {\sf 0}_4\\
-{\sf J_4} & {\sf 0}_4 & L_4 & {\sf 0}_4\\
-{\sf J_4} & {\sf 0}_4 & {\sf 0}_4 & L_4
\end{bmatrix},
\]respectively. 
Therefore, $A_4(C_4^{\mathbf{r}},t)\subseteq A_4\left(C_4^{(-4,-4,-4,-4)},t\right)$ and $A_4(S_4^{\mathbf{r}},t)\subseteq A_4\left(S_4^{(-4,-4,-4,-4)},t\right)$ for every $-{\bf r}$ such that $\mathbf{r}\in\mathbb{N}^4$. 
\end{obs}

\begin{lemma}\label{lemma:gb_Cr4}
Let ${\bf r}$ such that $-\mathbf{r}\in\mathbb{N}^4$. Then the fourth characteristic ideal of $C_4^{\mathbf{r}}$ is not trivial. Moreover, $A_4(C_4^{\mathbf{r}},t)\subseteq \langle t+1, 3\rangle$.
\end{lemma}
\begin{proof}
The Gr\"obner basis of the ideal generated by the 4-minors of the matrix $t{\sf I_{16}}-A\left(C_4^{(-4,-4,-4,-4)}\right)$ is $\langle t+1, 3\rangle$, that is, $A_4\left(C_4^{(-4,-4,-4,-4)},t\right)=\langle t+1, 3\rangle$. Hence, by the argument in Observation \ref{obs: 4-th minors}, we have that $A_4(C_4^{\mathbf{r}},t) \subseteq \langle t+1, 3\rangle$ for any ${\bf r}$ such that $-\mathbf{r}\in\mathbb{N}^4$.
\end{proof}

Note that $A_4(C_4^{\mathbf{r}},t)= \langle t+1, 3\rangle$ and $A_3(C_4^{\mathbf{r}},t)= \langle 1\rangle$ when $\mathbf{r}_1,\mathbf{r}_2,\mathbf{r}_3,\mathbf{r}_1\leq-4$.
In a similar manner, given that the Gr\"obner basis of $A_4\left(S_4^{(-4,-4,-4,-4)},t\right)$ is $\langle t+1, 2\rangle$, we have the following

\begin{lemma}\label{lemma:gb_Sr4}
Let ${\bf r}$ such that $-\mathbf{r}\in\mathbb{N}^4$. Then the fourth characteristic ideal of $S_4^{\mathbf{r}}$ is not trivial.
Moreover, $A_4(S_4^{\mathbf{r}},t)\subseteq \langle t+1, 2\rangle$.
\end{lemma}

\begin{theorem}
A connected graph $G$ is in $\mathcal{C}_{\leq 3}$ if and only if it is an induced subgraph of one of the following:
\begin{enumerate}
\item[\textup{(1)}] $C_5$,
\item[\textup{(2)}] the triangular prism $K_3 \Box K_2$,
\item[\textup{(3)}] a complete $4$-partite graph,
\item[\textup{(4)}]  $C_4^{\mathbf{r}}$, for some ${\bf r}$ such that $-\mathbf{r}\in \mathbb{N}^4$, or
\item[\textup{(5)}] $S_4^{\mathbf{r}}$, for some ${\bf r}$ such that $-\mathbf{r}\in \mathbb{N}^4$.
\end{enumerate}
\end{theorem}
\begin{proof}
$\Rightarrow)$ This follows from Lemma~\ref{lemma:GinC_leq3isFfree} and Theorem~\ref{teo:charC_leq3.1}.

$\Leftarrow)$ From Lemmas \ref{lemma:CharCompleteMultipartite}, \ref{lemma:gb_C5_K3K2},  \ref{lemma:gb_Cr4} and \ref{lemma:gb_Sr4}, we have that the $4${\it -th} characteristic ideals of the graphs $C_5$, $K_3 \Box K_2$, complete $4$-partite graphs, $C^{\bf r}_4$ and $S^{\bf r}_4$ are not trivial.
Then, by Lemma \ref{lemma:CharofInducedsubgraphs}, the $4${\it -th} characteristic ideal of any induced subgraph of these graphs is non-trivial.
\end{proof}

Now, we give the characterization of the regular graphs whose critical group has at most 3 invariant factors equal to 1.

\begin{corollary}
    Let $G$ be a connected simple regular graph. Then $G\in \mathcal{K}_{\leq 3}$ if and only if $G$ is one of the following:
    \begin{itemize}
        \item[(a)] $C_5$,
        \item[(b)] $K_3 \Box K_2$, 
        \item[(c)] a complete graph $K_r$,
        \item[(d)] a regular complete bipartite graph $K_{r,r}$,
        \item[(e)] a regular complete tripartite graph $K_{r,r,r}$,
        \item[(f)] a regular complete graph $4$-partite graph $K_{r,r,r,r}$,
        \item[(g)] $C_4^{(-r,-r,-r,-r)}$, for any $r\in \mathbb{N}$.
    \end{itemize}
\end{corollary}
\begin{proof}
$\Rightarrow)$ This follows from the fact that $G$ is a regular graph in $\mathcal{C}_{\leq 3}$.

$\Leftarrow)$ It is clear from Lemma \ref{lemma:gb_C5_K3K2} that the fourth invariant factors of $C_5$ and $K_3 \Box K_2$ are different than $1$.
The graphs in (c), (d) and (e) are precisely the graphs in $\mathcal{K}_{\leq 2}\subset \mathcal{K}_{\leq 3}$. For (f), if $r=1$ we have the complete graph with four vertices. Therefore, we assume that $r\geq 2$. By Lemma \ref{lemma:CharCompleteMultipartite}, the fourth characteristic ideal of a $4$-partite regular complete graph is
$\langle 3,t \rangle$ and its third characteristic ideal is trivial. Then, evaluating at $t=3r$, the degree of any vertex, we have that the fourth invariant factor is $\gcd(3,3r)=3$ and therefore $K_{r,r,r,r}\in \mathcal{K}_{\leq 3}$. Finally, for (g), note that the degree of any vertex of $C_4^{(-r,-r,-r,-r)}$ is $3r-1$. 
By Lemma~\ref{lemma:gb_Cr4}, when $r\geq 4$ the fourth invariant factor is the $\gcd(3r,3)=3$.
Thus $C_4^{(-r,-r,-r,-r)}\in \mathcal{K}_{\geq3}$ when $r\geq4$.
The lower cases can be explicitly computed to verify that $C_4^{(-r,-r,-r,-r)}\in \mathcal{K}_{\leq3}$. 
\end{proof}

\section{Graphs whose Smith group has at most 4 invariant factors equal to 1}\label{sec:smithgroup}

In this section we give the characterizations of the graph families $\mathcal{S}_{\leq k}$ for $k\in\{1,2,3\}$.
And for $k=4$, we give a set of 43 minimal forbidden graphs for $\mathcal{S}_{\leq 4}$.


We have that $\mathcal{S}_{\leq k}$ is closed under induced subgraphs.
This observation follows from next proposition.

\begin{proposition}\label{prop:betainducedsubgraph}
  If $H$ is an induced subgraph of $G$, then $\phi(A(H))\leq\phi(A(G))$.
\end{proposition}
\begin{proof}
Let $H$ be an induced subgraph of $G$.
For any $k$ such that $1\leq k\leq |V(H)|$, the $k$-minors of $A(H)$ are contained in the $k$-minors of $A(G)$.
Therefore, if $\Delta_k(A(H))=1$, then $\Delta_k(A(G))=1$.
\end{proof}

Given a family $\mathcal F$ of graphs, a graph $G$ is called $\mathcal F$-free if no induced subgraph of $G$ is isomorphic to a member of $\mathcal F$.
We can define a graph $G$ to be \textit{forbidden} for $\mathcal{S}_{\leq k}$ when $\phi(A(G))\geq k+1$.
Let ${\bf Forb}(\mathcal{S}_{\leq k})$ denote the set of minimal forbidden graphs for $\mathcal{S}_{\leq k}$ with respect to the induced subgraph order.
Thus $G\in\mathcal{S}_{\leq k}$ if and only if $G$ is ${\bf Forb}(\mathcal{S}_{\leq 2})$-free.
Therefore, characterizing the minimal forbidden induced subgraphs for $\mathcal S_{\leq k}$ leads to a characterization of $\mathcal{S}_{\leq k}$.
For instance, let $P_2$ denote the path with 2 vertices.
We have that the Smith normal form of the adjacency matrix of $P_2$ has 2 invariant factors equal to 1.
Then, $\mathcal{S}_{\leq 1}$ consists only of $K_1$, and there is no graph $G$ with $\phi(A(G))=1$.

\begin{figure}[h!]
\begin{center}
    \begin{tikzpicture}[scale=1,thick]
    \tikzstyle{every node}=[minimum width=0pt, inner sep=2pt, circle]
        \draw (0:1) node[draw] (0) {};
        \draw (0:0) node[draw] (1) {};
        \draw (150:1) node[draw] (2) {};
        \draw (210:1) node[draw] (3) {};
        \draw  (0) -- (1) -- (2) -- (3) -- (1);
    \end{tikzpicture}
\end{center}
\caption{{\sf paw} graph}
\label{figure:paw}
\end{figure}
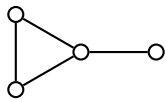

Now, we are going to give an alternative proof of the characterization $\mathcal{S}_{\leq 2}$.
For this, next result gives the SNF for complete $k$-partite graphs.
A particular case of Lemma~\ref{lemma:CharCompleteMultipartite} is the following lemma, which also was noticed in \cite{BK}.

\begin{lemma}\label{lemma:SNFcompletekpartite}
  Let $G$ be a complete $k$-partite graph with $n$ vertices.
  Then the Smith normal form of $A(G)$ is equal to $I_{k-1}\oplus (k-1)\oplus 0I_{n-k}$.
\end{lemma}

Now we are ready to give the characterization of graphs whose Smith group have at most 2 invariant factors.

\begin{theorem}\label{teo:graphsSmithgroup2if}
  Let $G$ be connected graph.
  Then the followings are equivalent.
  \begin{enumerate}
    \item the SNF of $A(G)$ has at most 2 invariant factors equal to 1,
    \item $G$ is $\{P_4,{\sf paw},K_4\}$-free,
    \item $G$ is an induced subgraph of a complete tripartite graph.
  \end{enumerate}
\end{theorem}
\begin{proof}
$(1)\implies (2)$ The SNF of the adjacency matrices of $P_4$, {\sf paw}, $K_4$ are equal to $\diag(1,1,1,1)$, $\diag(1,1,1,1)$ and $\diag(1,1,1,3)$, respectively.
Since any induced subgraph $H$ of $P_4$, {\sf paw}, or $K_4$ has $\phi(A(H))\leq2$, then $\{P_4,{\sf paw},K_4\}\subseteq{\bf Forb}(\mathcal{S}_{\leq 2})$.

$(2)\implies (3)$ By Lemma~\ref{lemma:pawfree}, $G$ is either triangle free or a complete multipartite graph.
In the first case by Lemma~\ref{lemma:P4pawfree}, $G$ is a complete bipartite graph.
And in the second case since $G$ is $K_4$-free, then $G$ is complete tripartite graph.

$(3)\implies (1)$ It follows by Lemma~\ref{lemma:SNFcompletekpartite} that the SNF of $A(G)$  is at most 2.
\end{proof}

An analogous reasoning give us the characterization of $\mathcal{S}_{\leq 3}$.

\begin{theorem}\label{teo:graphsSmithgroup3if}
  Let $G$ be connected graph.
  Then the followings are equivalent.
  \begin{enumerate}
    \item the SNF of $A(G)$ has at most 3 invariant factors equal to 1
    \item $G$ is $\{P_4,{\sf paw},K_5\}$-free
    \item $G$ is an induced subgraph of a complete four-partite graph
  \end{enumerate}
\end{theorem}

\begin{proof}
$(1)\implies (2)$ The SNF of the adjacency matrices of $P_4$, {\sf paw}, $K_5$ are equal to $\diag(1,1,1,1)$, $\diag(1,1,1,1)$ and $\diag(1,1,1,1,4)$, respectively.
Since any induced subgraph $H$ of $P_4$, {\sf paw}, or $K_5$ has $\phi(A(H))\leq3$, then $\{P_4,{\sf paw},K_5\}\subseteq{\bf Forb}(\mathcal{S}_{\leq 3})$.

$(2)\implies (3)$ Since $G$ is {\sf paw}-free, then by Lemma~\ref{lemma:pawfree}, $G$ is either triangle-free or a complete multipartite graph.
Thus, in the first case, $G$ is also $K_3$-free, by Lemma~\ref{lemma:P4pawfree}, $G$ is a complete bipartite graph.
In the second case, since $G$ is $K_5$-free, then $G$ is complete tripartite graph.

$(3)\implies (1)$ It follows by Lemma~\ref{lemma:SNFcompletekpartite} that the SNF of $A(G)$  is at most 3.
\end{proof}

The next case is more complicated.
With the use of SAGE \cite{sage}, we found that there are 43 forbidden graphs for $\mathcal{S}_{\leq4}$, see Figure \ref{fig:forb4}.
The following SAGE code computes the minimal forbidden graphs with at most $m$ vertices for $\mathcal{S}_{\leq n}$.

\begin{lstlisting}
import numpy as np

def MinForb(m,n):
	Forbidden = []
	for k in range(2,m):
		for g in graphs(k):
			if g.is_connected():
				SNF = g.adjacency_matrix().smith_form()[0].numpy().diagonal()
				num_ones = list(SNF).count(1)
				if num_ones >= n+1:
					Forbidden.append([g.graph6_string(),num_ones])

	Minimal = []
	for g in range(0,len(Forbidden)):
		flag = True
		for h in Minimal:
			if Graph(Forbidden[g][0]).subgraph_search(Graph(h),induced=True) != None:
				flag = False
				break
		if flag == True:
			Minimal.append(Forbidden[g][0])
	return Minimal
\end{lstlisting}


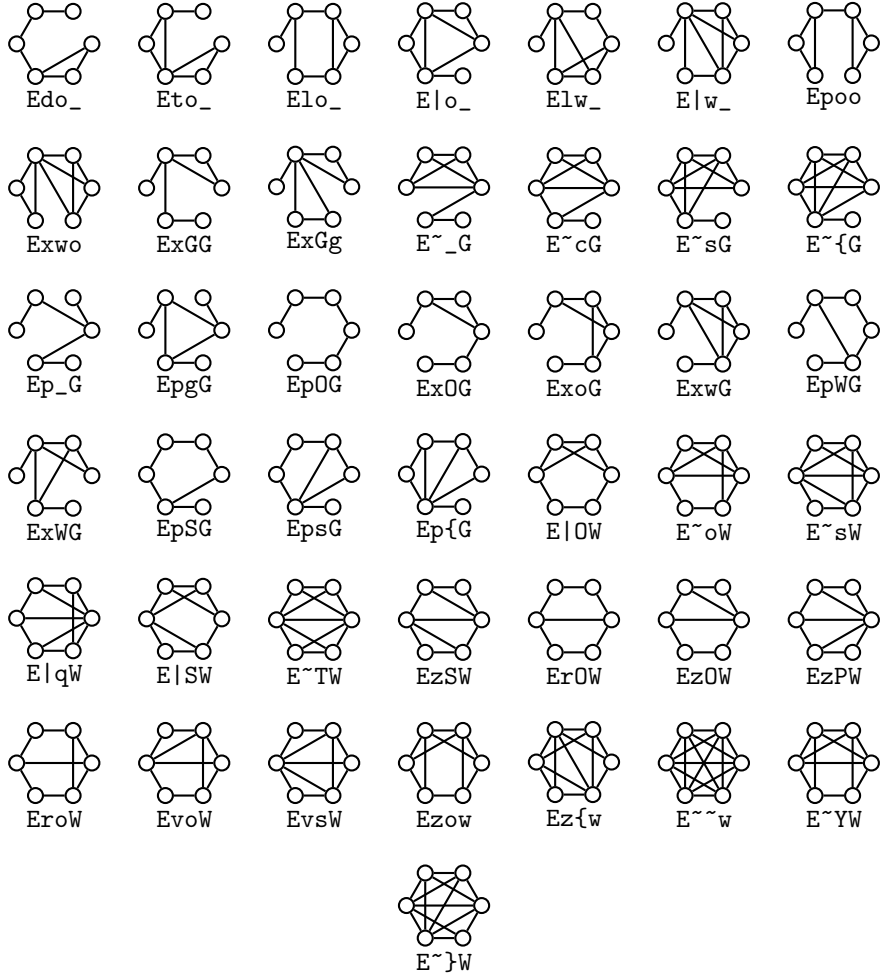
\begin{figure}[h!]

\begin{center}

\begin{tabular}{c@{\extracolsep{.5cm}}c@{\extracolsep{.5cm}}c@{\extracolsep{.5cm}}c@{\extracolsep{.5cm}}c@{\extracolsep{.5cm}}c@{\extracolsep{.5cm}}c}
    \begin{tikzpicture}[scale=.5,thick]
    \tikzstyle{every node}=[minimum width=0pt, inner sep=2pt, circle]
        \draw (0:1) node[draw] (4) { };
        \draw (60:1) node[draw] (5) { };
        \draw (120:1) node[draw] (2) { };
        \draw (180:1) node[draw] (3) { };
        \draw (240:1) node[draw] (0) { };
        \draw (300:1) node[draw] (1) { };
        \draw  (0) edge (1);
        \draw  (0) edge (3);
        \draw  (0) edge (4);
        \draw  (1) edge (4);
        \draw  (2) edge (3);
        \draw  (2) edge (5);
    \draw (0,-1.5) node () { \text{\texttt{Edo{\char`\_}}} };
    \end{tikzpicture}
&
    \begin{tikzpicture}[scale=.5,thick]
    \tikzstyle{every node}=[minimum width=0pt, inner sep=2pt, circle]
        \draw (0:1) node[draw] (4) { };
        \draw (60:1) node[draw] (5) { };
        \draw (120:1) node[draw] (2) { };
        \draw (180:1) node[draw] (3) { };
        \draw (240:1) node[draw] (0) { };
        \draw (300:1) node[draw] (1) { };
        \draw  (0) edge (1);
        \draw  (0) edge (2);
        \draw  (0) edge (3);
        \draw  (0) edge (4);
        \draw  (1) edge (4);
        \draw  (2) edge (3);
        \draw  (2) edge (5);
    \draw (0,-1.5) node () { \text{\texttt{Eto{\char`\_}}} };
    \end{tikzpicture}
&
    \begin{tikzpicture}[scale=.5,thick]
    \tikzstyle{every node}=[minimum width=0pt, inner sep=2pt, circle]
        \draw (0:1) node[draw] (4) { };
        \draw (60:1) node[draw] (1) { };
        \draw (120:1) node[draw] (2) { };
        \draw (180:1) node[draw] (5) { };
        \draw (240:1) node[draw] (3) { };
        \draw (300:1) node[draw] (0) { };
        \draw  (0) edge (1);
        \draw  (0) edge (3);
        \draw  (0) edge (4);
        \draw  (1) edge (2);
        \draw  (1) edge (4);
        \draw  (2) edge (3);
        \draw  (2) edge (5);
    \draw (0,-1.5) node () { \text{\texttt{Elo{\char`\_}}} };
    \end{tikzpicture}
&
    \begin{tikzpicture}[scale=.5,thick]
    \tikzstyle{every node}=[minimum width=0pt, inner sep=2pt, circle]
        \draw (0:1) node[draw] (1) { };
        \draw (60:1) node[draw] (4) { };
        \draw (120:1) node[draw] (0) { };
        \draw (180:1) node[draw] (3) { };
        \draw (240:1) node[draw] (2) { };
        \draw (300:1) node[draw] (5) { };
        \draw  (0) edge (1);
        \draw  (0) edge (2);
        \draw  (0) edge (3);
        \draw  (0) edge (4);
        \draw  (1) edge (2);
        \draw  (1) edge (4);
        \draw  (2) edge (3);
        \draw  (2) edge (5);
    \draw (0,-1.5) node () { \text{\texttt{E|o{\char`\_}}} };
    \end{tikzpicture}
&
    \begin{tikzpicture}[scale=.5,thick]
    \tikzstyle{every node}=[minimum width=0pt, inner sep=2pt, circle]
        \draw (0:1) node[draw] (0) { };
        \draw (60:1) node[draw] (3) { };
        \draw (120:1) node[draw] (2) { };
        \draw (180:1) node[draw] (5) { };
        \draw (240:1) node[draw] (1) { };
        \draw (300:1) node[draw] (4) { };
        \draw  (0) edge (1);
        \draw  (0) edge (3);
        \draw  (0) edge (4);
        \draw  (1) edge (2);
        \draw  (1) edge (4);
        \draw  (2) edge (3);
        \draw  (2) edge (4);
        \draw  (2) edge (5);
    \draw (0,-1.5) node () { \text{\texttt{Elw{\char`\_}}} };
    \end{tikzpicture}
&
    \begin{tikzpicture}[scale=.5,thick]
    \tikzstyle{every node}=[minimum width=0pt, inner sep=2pt, circle]
        \draw (0:1) node[draw] (4) { };
        \draw (60:1) node[draw] (1) { };
        \draw (120:1) node[draw] (2) { };
        \draw (180:1) node[draw] (5) { };
        \draw (240:1) node[draw] (3) { };
        \draw (300:1) node[draw] (0) { };
        \draw  (0) edge (1);
        \draw  (0) edge (2);
        \draw  (0) edge (3);
        \draw  (0) edge (4);
        \draw  (1) edge (2);
        \draw  (1) edge (4);
        \draw  (2) edge (3);
        \draw  (2) edge (4);
        \draw  (2) edge (5);
    \draw (0,-1.5) node () { \text{\texttt{E|w{\char`\_}}} };
    \end{tikzpicture}
&
    \begin{tikzpicture}[scale=.5,thick]
    \tikzstyle{every node}=[minimum width=0pt, inner sep=2pt, circle]
        \draw (0:1) node[draw] (1) { };
        \draw (60:1) node[draw] (0) { };
        \draw (120:1) node[draw] (2) { };
        \draw (180:1) node[draw] (3) { };
        \draw (240:1) node[draw] (5) { };
        \draw (300:1) node[draw] (4) { };
        \draw  (0) edge (1);
        \draw  (0) edge (2);
        \draw  (0) edge (4);
        \draw  (1) edge (4);
        \draw  (2) edge (3);
        \draw  (2) edge (5);
        \draw  (3) edge (5);
    \draw (0,-1.5) node () { \text{\texttt{Epoo}} };
    \end{tikzpicture}
\\
    \begin{tikzpicture}[scale=.5,thick]
    \tikzstyle{every node}=[minimum width=0pt, inner sep=2pt, circle]
        \draw (0:1) node[draw] (0) { };
        \draw (60:1) node[draw] (1) { };
        \draw (120:1) node[draw] (2) { };
        \draw (180:1) node[draw] (3) { };
        \draw (240:1) node[draw] (5) { };
        \draw (300:1) node[draw] (4) { };
        \draw  (0) edge (1);
        \draw  (0) edge (2);
        \draw  (0) edge (4);
        \draw  (1) edge (2);
        \draw  (1) edge (4);
        \draw  (2) edge (3);
        \draw  (2) edge (4);
        \draw  (2) edge (5);
        \draw  (3) edge (5);
    \draw (0,-1.5) node () { \text{\texttt{Exwo}} };
    \end{tikzpicture}
&
    \begin{tikzpicture}[scale=.5,thick]
    \tikzstyle{every node}=[minimum width=0pt, inner sep=2pt, circle]
        \draw (0:1) node[draw] (0) { };
        \draw (60:1) node[draw] (1) { };
        \draw (120:1) node[draw] (2) { };
        \draw (180:1) node[draw] (3) { };
        \draw (240:1) node[draw] (4) { };
        \draw (300:1) node[draw] (5) { };
        \draw  (0) edge (1);
        \draw  (0) edge (2);
        \draw  (1) edge (2);
        \draw  (2) edge (3);
        \draw  (2) edge (4);
        \draw  (4) edge (5);
    \draw (0,-1.5) node () { \text{\texttt{ExGG}} };
    \end{tikzpicture}
&
    \begin{tikzpicture}[scale=.5,thick]
    \tikzstyle{every node}=[minimum width=0pt, inner sep=2pt, circle]
        \draw (0:1) node[draw] (0) { };
        \draw (60:1) node[draw] (1) { };
        \draw (120:1) node[draw] (2) { };
        \draw (180:1) node[draw] (3) { };
        \draw (240:1) node[draw] (4) { };
        \draw (300:1) node[draw] (5) { };
        \draw  (0) edge (1);
        \draw  (0) edge (2);
        \draw  (1) edge (2);
        \draw  (2) edge (3);
        \draw  (2) edge (4);
        \draw  (2) edge (5);
        \draw  (4) edge (5);
    \draw (0,-1.5) node () { \text{\texttt{ExGg}} };
    \end{tikzpicture}
&
    \begin{tikzpicture}[scale=.5,thick]
    \tikzstyle{every node}=[minimum width=0pt, inner sep=2pt, circle]
        \draw (0:1) node[draw] (0) { };
        \draw (60:1) node[draw] (1) { };
        \draw (120:1) node[draw] (2) { };
        \draw (180:1) node[draw] (3) { };
        \draw (240:1) node[draw] (4) { };
        \draw (300:1) node[draw] (5) { };
        \draw  (0) edge (1);
        \draw  (0) edge (2);
        \draw  (0) edge (3);
        \draw  (0) edge (4);
        \draw  (1) edge (2);
        \draw  (1) edge (3);
        \draw  (2) edge (3);
        \draw  (4) edge (5);
    \draw (0,-1.5) node () { \text{\texttt{E{\char`\~}{\char`\_}G}} };
    \end{tikzpicture}
&
    \begin{tikzpicture}[scale=.5,thick]
    \tikzstyle{every node}=[minimum width=0pt, inner sep=2pt, circle]
        \draw (0:1) node[draw] (0) { };
        \draw (60:1) node[draw] (1) { };
        \draw (120:1) node[draw] (2) { };
        \draw (180:1) node[draw] (3) { };
        \draw (240:1) node[draw] (4) { };
        \draw (300:1) node[draw] (5) { };
        \draw  (0) edge (1);
        \draw  (0) edge (2);
        \draw  (0) edge (3);
        \draw  (0) edge (4);
        \draw  (1) edge (2);
        \draw  (1) edge (3);
        \draw  (2) edge (3);
        \draw  (3) edge (4);
        \draw  (4) edge (5);
    \draw (0,-1.5) node () { \text{\texttt{E{\char`\~}cG}} };
    \end{tikzpicture}
&
    \begin{tikzpicture}[scale=.5,thick]
    \tikzstyle{every node}=[minimum width=0pt, inner sep=2pt, circle]
        \draw (0:1) node[draw] (2) { };
        \draw (60:1) node[draw] (1) { };
        \draw (120:1) node[draw] (0) { };
        \draw (180:1) node[draw] (3) { };
        \draw (240:1) node[draw] (4) { };
        \draw (300:1) node[draw] (5) { };
        \draw  (0) edge (1);
        \draw  (0) edge (2);
        \draw  (0) edge (3);
        \draw  (0) edge (4);
        \draw  (1) edge (2);
        \draw  (1) edge (3);
        \draw  (1) edge (4);
        \draw  (2) edge (3);
        \draw  (3) edge (4);
        \draw  (4) edge (5);
    \draw (0,-1.5) node () { \text{\texttt{E{\char`\~}sG}} };
    \end{tikzpicture}
&
    \begin{tikzpicture}[scale=.5,thick]
    \tikzstyle{every node}=[minimum width=0pt, inner sep=2pt, circle]
        \draw (0:1) node[draw] (0) { };
        \draw (60:1) node[draw] (1) { };
        \draw (120:1) node[draw] (2) { };
        \draw (180:1) node[draw] (3) { };
        \draw (240:1) node[draw] (4) { };
        \draw (300:1) node[draw] (5) { };
        \draw  (0) edge (1);
        \draw  (0) edge (2);
        \draw  (0) edge (3);
        \draw  (0) edge (4);
        \draw  (1) edge (2);
        \draw  (1) edge (3);
        \draw  (1) edge (4);
        \draw  (2) edge (3);
        \draw  (2) edge (4);
        \draw  (3) edge (4);
        \draw  (4) edge (5);
    \draw (0,-1.5) node () { \text{\texttt{E{\char`\~}{\char`\{}G}} };
    \end{tikzpicture}
\\
    \begin{tikzpicture}[scale=.5,thick]
    \tikzstyle{every node}=[minimum width=0pt, inner sep=2pt, circle]
        \draw (0:1) node[draw] (0) { };
        \draw (60:1) node[draw] (1) { };
        \draw (120:1) node[draw] (2) { };
        \draw (180:1) node[draw] (3) { };
        \draw (240:1) node[draw] (4) { };
        \draw (300:1) node[draw] (5) { };
        \draw  (0) edge (1);
        \draw  (0) edge (2);
        \draw  (0) edge (4);
        \draw  (2) edge (3);
        \draw  (4) edge (5);
    \draw (0,-1.5) node () { \text{\texttt{Ep{\char`\_}G}} };
    \end{tikzpicture}
&
    \begin{tikzpicture}[scale=.5,thick]
    \tikzstyle{every node}=[minimum width=0pt, inner sep=2pt, circle]
        \draw (0:1) node[draw] (0) { };
        \draw (60:1) node[draw] (1) { };
        \draw (120:1) node[draw] (2) { };
        \draw (180:1) node[draw] (3) { };
        \draw (240:1) node[draw] (4) { };
        \draw (300:1) node[draw] (5) { };
        \draw  (0) edge (1);
        \draw  (0) edge (2);
        \draw  (0) edge (4);
        \draw  (2) edge (3);
        \draw  (2) edge (4);
        \draw  (4) edge (5);
    \draw (0,-1.5) node () { \text{\texttt{EpgG}} };
    \end{tikzpicture}
&
    \begin{tikzpicture}[scale=.5,thick]
    \tikzstyle{every node}=[minimum width=0pt, inner sep=2pt, circle]
        \draw (0:1) node[draw] (1) { };
        \draw (60:1) node[draw] (0) { };
        \draw (120:1) node[draw] (2) { };
        \draw (180:1) node[draw] (3) { };
        \draw (240:1) node[draw] (5) { };
        \draw (300:1) node[draw] (4) { };
        \draw  (0) edge (1);
        \draw  (0) edge (2);
        \draw  (1) edge (4);
        \draw  (2) edge (3);
        \draw  (4) edge (5);
    \draw (0,-1.5) node () { \text{\texttt{EpOG}} };
    \end{tikzpicture}
&
    \begin{tikzpicture}[scale=.5,thick]
    \tikzstyle{every node}=[minimum width=0pt, inner sep=2pt, circle]
        \draw (0:1) node[draw] (1) { };
        \draw (60:1) node[draw] (0) { };
        \draw (120:1) node[draw] (2) { };
        \draw (180:1) node[draw] (3) { };
        \draw (240:1) node[draw] (5) { };
        \draw (300:1) node[draw] (4) { };
        \draw  (0) edge (1);
        \draw  (0) edge (2);
        \draw  (1) edge (2);
        \draw  (1) edge (4);
        \draw  (2) edge (3);
        \draw  (4) edge (5);
    \draw (0,-1.5) node () { \text{\texttt{ExOG}} };
    \end{tikzpicture}
&
    \begin{tikzpicture}[scale=.5,thick]
    \tikzstyle{every node}=[minimum width=0pt, inner sep=2pt, circle]
        \draw (0:1) node[draw] (0) { };
        \draw (60:1) node[draw] (1) { };
        \draw (120:1) node[draw] (2) { };
        \draw (180:1) node[draw] (3) { };
        \draw (240:1) node[draw] (5) { };
        \draw (300:1) node[draw] (4) { };
        \draw  (0) edge (1);
        \draw  (0) edge (2);
        \draw  (0) edge (4);
        \draw  (1) edge (2);
        \draw  (1) edge (4);
        \draw  (2) edge (3);
        \draw  (4) edge (5);
    \draw (0,-1.5) node () { \text{\texttt{ExoG}} };
    \end{tikzpicture}
&
    \begin{tikzpicture}[scale=.5,thick]
    \tikzstyle{every node}=[minimum width=0pt, inner sep=2pt, circle]
        \draw (0:1) node[draw] (0) { };
        \draw (60:1) node[draw] (1) { };
        \draw (120:1) node[draw] (2) { };
        \draw (180:1) node[draw] (3) { };
        \draw (240:1) node[draw] (5) { };
        \draw (300:1) node[draw] (4) { };
        \draw  (0) edge (1);
        \draw  (0) edge (2);
        \draw  (0) edge (4);
        \draw  (1) edge (2);
        \draw  (1) edge (4);
        \draw  (2) edge (3);
        \draw  (2) edge (4);
        \draw  (4) edge (5);
    \draw (0,-1.5) node () { \text{\texttt{ExwG}} };
    \end{tikzpicture}
&
    \begin{tikzpicture}[scale=.5,thick]
    \tikzstyle{every node}=[minimum width=0pt, inner sep=2pt, circle]
        \draw (0:1) node[draw] (1) { };
        \draw (60:1) node[draw] (0) { };
        \draw (120:1) node[draw] (2) { };
        \draw (180:1) node[draw] (3) { };
        \draw (240:1) node[draw] (5) { };
        \draw (300:1) node[draw] (4) { };
        \draw  (0) edge (1);
        \draw  (0) edge (2);
        \draw  (1) edge (4);
        \draw  (2) edge (3);
        \draw  (2) edge (4);
        \draw  (4) edge (5);
    \draw (0,-1.5) node () { \text{\texttt{EpWG}} };
    \end{tikzpicture}
\\
    \begin{tikzpicture}[scale=.5,thick]
    \tikzstyle{every node}=[minimum width=0pt, inner sep=2pt, circle]
        \draw (0:1) node[draw] (0) { };
        \draw (60:1) node[draw] (1) { };
        \draw (120:1) node[draw] (2) { };
        \draw (180:1) node[draw] (3) { };
        \draw (240:1) node[draw] (4) { };
        \draw (300:1) node[draw] (5) { };
        \draw  (0) edge (1);
        \draw  (0) edge (2);
        \draw  (1) edge (2);
        \draw  (1) edge (4);
        \draw  (2) edge (3);
        \draw  (2) edge (4);
        \draw  (4) edge (5);
    \draw (0,-1.5) node () { \text{\texttt{ExWG}} };
    \end{tikzpicture}
&
    \begin{tikzpicture}[scale=.5,thick]
    \tikzstyle{every node}=[minimum width=0pt, inner sep=2pt, circle]
        \draw (0:1) node[draw] (1) { };
        \draw (60:1) node[draw] (0) { };
        \draw (120:1) node[draw] (2) { };
        \draw (180:1) node[draw] (3) { };
        \draw (240:1) node[draw] (4) { };
        \draw (300:1) node[draw] (5) { };
        \draw  (0) edge (1);
        \draw  (0) edge (2);
        \draw  (1) edge (4);
        \draw  (2) edge (3);
        \draw  (3) edge (4);
        \draw  (4) edge (5);
    \draw (0,-1.5) node () { \text{\texttt{EpSG}} };
    \end{tikzpicture}
&
    \begin{tikzpicture}[scale=.5,thick]
    \tikzstyle{every node}=[minimum width=0pt, inner sep=2pt, circle]
        \draw (0:1) node[draw] (1) { };
        \draw (60:1) node[draw] (0) { };
        \draw (120:1) node[draw] (2) { };
        \draw (180:1) node[draw] (3) { };
        \draw (240:1) node[draw] (4) { };
        \draw (300:1) node[draw] (5) { };
        \draw  (0) edge (1);
        \draw  (0) edge (2);
        \draw  (0) edge (4);
        \draw  (1) edge (4);
        \draw  (2) edge (3);
        \draw  (3) edge (4);
        \draw  (4) edge (5);
    \draw (0,-1.5) node () { \text{\texttt{EpsG}} };
    \end{tikzpicture}
&
    \begin{tikzpicture}[scale=.5,thick]
    \tikzstyle{every node}=[minimum width=0pt, inner sep=2pt, circle]
        \draw (0:1) node[draw] (1) { };
        \draw (60:1) node[draw] (0) { };
        \draw (120:1) node[draw] (2) { };
        \draw (180:1) node[draw] (3) { };
        \draw (240:1) node[draw] (4) { };
        \draw (300:1) node[draw] (5) { };
        \draw  (0) edge (1);
        \draw  (0) edge (2);
        \draw  (0) edge (4);
        \draw  (1) edge (4);
        \draw  (2) edge (3);
        \draw  (2) edge (4);
        \draw  (3) edge (4);
        \draw  (4) edge (5);
    \draw (0,-1.5) node () { \text{\texttt{Ep{\char`\{}G}} };
    \end{tikzpicture}
&
    \begin{tikzpicture}[scale=.5,thick]
    \tikzstyle{every node}=[minimum width=0pt, inner sep=2pt, circle]
        \draw (0:1) node[draw] (1) { };
        \draw (60:1) node[draw] (0) { };
        \draw (120:1) node[draw] (2) { };
        \draw (180:1) node[draw] (3) { };
        \draw (240:1) node[draw] (5) { };
        \draw (300:1) node[draw] (4) { };
        \draw  (0) edge (1);
        \draw  (0) edge (2);
        \draw  (0) edge (3);
        \draw  (1) edge (2);
        \draw  (1) edge (4);
        \draw  (2) edge (3);
        \draw  (3) edge (5);
        \draw  (4) edge (5);
    \draw (0,-1.5) node () { \text{\texttt{E|OW}} };
    \end{tikzpicture}
&
    \begin{tikzpicture}[scale=.5,thick]
    \tikzstyle{every node}=[minimum width=0pt, inner sep=2pt, circle]
        \draw (0:1) node[draw] (0) { };
        \draw (60:1) node[draw] (1) { };
        \draw (120:1) node[draw] (2) { };
        \draw (180:1) node[draw] (3) { };
        \draw (240:1) node[draw] (5) { };
        \draw (300:1) node[draw] (4) { };
        \draw  (0) edge (1);
        \draw  (0) edge (2);
        \draw  (0) edge (3);
        \draw  (0) edge (4);
        \draw  (1) edge (2);
        \draw  (1) edge (3);
        \draw  (1) edge (4);
        \draw  (2) edge (3);
        \draw  (3) edge (5);
        \draw  (4) edge (5);
    \draw (0,-1.5) node () { \text{\texttt{E{\char`\~}oW}} };
    \end{tikzpicture}
&
    \begin{tikzpicture}[scale=.5,thick]
    \tikzstyle{every node}=[minimum width=0pt, inner sep=2pt, circle]
        \draw (0:1) node[draw] (0) { };
        \draw (60:1) node[draw] (1) { };
        \draw (120:1) node[draw] (2) { };
        \draw (180:1) node[draw] (3) { };
        \draw (240:1) node[draw] (5) { };
        \draw (300:1) node[draw] (4) { };
        \draw  (0) edge (1);
        \draw  (0) edge (2);
        \draw  (0) edge (3);
        \draw  (0) edge (4);
        \draw  (1) edge (2);
        \draw  (1) edge (3);
        \draw  (1) edge (4);
        \draw  (2) edge (3);
        \draw  (3) edge (4);
        \draw  (3) edge (5);
        \draw  (4) edge (5);
    \draw (0,-1.5) node () { \text{\texttt{E{\char`\~}sW}} };
    \end{tikzpicture}
\\
    \begin{tikzpicture}[scale=.5,thick]
    \tikzstyle{every node}=[minimum width=0pt, inner sep=2pt, circle]
        \draw (0:1) node[draw] (0) { };
        \draw (60:1) node[draw] (1) { };
        \draw (120:1) node[draw] (2) { };
        \draw (180:1) node[draw] (3) { };
        \draw (240:1) node[draw] (5) { };
        \draw (300:1) node[draw] (4) { };
        \draw  (0) edge (1);
        \draw  (0) edge (2);
        \draw  (0) edge (3);
        \draw  (0) edge (4);
        \draw  (0) edge (5);
        \draw  (1) edge (2);
        \draw  (1) edge (4);
        \draw  (2) edge (3);
        \draw  (3) edge (5);
        \draw  (4) edge (5);
    \draw (0,-1.5) node () { \text{\texttt{E|qW}} };
    \end{tikzpicture}
&
    \begin{tikzpicture}[scale=.5,thick]
    \tikzstyle{every node}=[minimum width=0pt, inner sep=2pt, circle]
        \draw (0:1) node[draw] (1) { };
        \draw (60:1) node[draw] (0) { };
        \draw (120:1) node[draw] (2) { };
        \draw (180:1) node[draw] (3) { };
        \draw (240:1) node[draw] (5) { };
        \draw (300:1) node[draw] (4) { };
        \draw  (0) edge (1);
        \draw  (0) edge (2);
        \draw  (0) edge (3);
        \draw  (1) edge (2);
        \draw  (1) edge (4);
        \draw  (2) edge (3);
        \draw  (3) edge (4);
        \draw  (3) edge (5);
        \draw  (4) edge (5);
    \draw (0,-1.5) node () { \text{\texttt{E|SW}} };
    \end{tikzpicture}
&
    \begin{tikzpicture}[scale=.5,thick]
    \tikzstyle{every node}=[minimum width=0pt, inner sep=2pt, circle]
        \draw (0:1) node[draw] (1) { };
        \draw (60:1) node[draw] (0) { };
        \draw (120:1) node[draw] (2) { };
        \draw (180:1) node[draw] (3) { };
        \draw (240:1) node[draw] (4) { };
        \draw (300:1) node[draw] (5) { };
        \draw  (0) edge (1);
        \draw  (0) edge (2);
        \draw  (0) edge (3);
        \draw  (1) edge (2);
        \draw  (1) edge (3);
        \draw  (1) edge (4);
        \draw  (1) edge (5);
        \draw  (2) edge (3);
        \draw  (3) edge (4);
        \draw  (3) edge (5);
        \draw  (4) edge (5);
    \draw (0,-1.5) node () { \text{\texttt{E{\char`\~}TW}} };
    \end{tikzpicture}
&
    \begin{tikzpicture}[scale=.5,thick]
    \tikzstyle{every node}=[minimum width=0pt, inner sep=2pt, circle]
        \draw (0:1) node[draw] (1) { };
        \draw (60:1) node[draw] (0) { };
        \draw (120:1) node[draw] (2) { };
        \draw (180:1) node[draw] (3) { };
        \draw (240:1) node[draw] (5) { };
        \draw (300:1) node[draw] (4) { };
        \draw  (0) edge (1);
        \draw  (0) edge (2);
        \draw  (1) edge (2);
        \draw  (1) edge (3);
        \draw  (1) edge (4);
        \draw  (2) edge (3);
        \draw  (3) edge (4);
        \draw  (3) edge (5);
        \draw  (4) edge (5);
    \draw (0,-1.5) node () { \text{\texttt{EzSW}} };
    \end{tikzpicture}
&
    \begin{tikzpicture}[scale=.5,thick]
    \tikzstyle{every node}=[minimum width=0pt, inner sep=2pt, circle]
        \draw (0:1) node[draw] (1) { };
        \draw (60:1) node[draw] (0) { };
        \draw (120:1) node[draw] (2) { };
        \draw (180:1) node[draw] (3) { };
        \draw (240:1) node[draw] (5) { };
        \draw (300:1) node[draw] (4) { };
        \draw  (0) edge (1);
        \draw  (0) edge (2);
        \draw  (1) edge (3);
        \draw  (1) edge (4);
        \draw  (2) edge (3);
        \draw  (3) edge (5);
        \draw  (4) edge (5);
    \draw (0,-1.5) node () { \text{\texttt{ErOW}} };
    \end{tikzpicture}
&
    \begin{tikzpicture}[scale=.5,thick]
    \tikzstyle{every node}=[minimum width=0pt, inner sep=2pt, circle]
        \draw (0:1) node[draw] (1) { };
        \draw (60:1) node[draw] (0) { };
        \draw (120:1) node[draw] (2) { };
        \draw (180:1) node[draw] (3) { };
        \draw (240:1) node[draw] (5) { };
        \draw (300:1) node[draw] (4) { };
        \draw  (0) edge (1);
        \draw  (0) edge (2);
        \draw  (1) edge (2);
        \draw  (1) edge (3);
        \draw  (1) edge (4);
        \draw  (2) edge (3);
        \draw  (3) edge (5);
        \draw  (4) edge (5);
    \draw (0,-1.5) node () { \text{\texttt{EzOW}} };
    \end{tikzpicture}
&
    \begin{tikzpicture}[scale=.5,thick]
    \tikzstyle{every node}=[minimum width=0pt, inner sep=2pt, circle]
        \draw (0:1) node[draw] (1) { };
        \draw (60:1) node[draw] (0) { };
        \draw (120:1) node[draw] (2) { };
        \draw (180:1) node[draw] (3) { };
        \draw (240:1) node[draw] (5) { };
        \draw (300:1) node[draw] (4) { };
        \draw  (0) edge (1);
        \draw  (0) edge (2);
        \draw  (1) edge (2);
        \draw  (1) edge (3);
        \draw  (1) edge (4);
        \draw  (1) edge (5);
        \draw  (2) edge (3);
        \draw  (3) edge (5);
        \draw  (4) edge (5);
    \draw (0,-1.5) node () { \text{\texttt{EzPW}} };
    \end{tikzpicture}
\\
    \begin{tikzpicture}[scale=.5,thick]
    \tikzstyle{every node}=[minimum width=0pt, inner sep=2pt, circle]
        \draw (0:1) node[draw] (1) { };
        \draw (60:1) node[draw] (0) { };
        \draw (120:1) node[draw] (2) { };
        \draw (180:1) node[draw] (3) { };
        \draw (240:1) node[draw] (5) { };
        \draw (300:1) node[draw] (4) { };
        \draw  (0) edge (1);
        \draw  (0) edge (2);
        \draw  (0) edge (4);
        \draw  (1) edge (3);
        \draw  (1) edge (4);
        \draw  (2) edge (3);
        \draw  (3) edge (5);
        \draw  (4) edge (5);
    \draw (0,-1.5) node () { \text{\texttt{EroW}} };
    \end{tikzpicture}
&
    \begin{tikzpicture}[scale=.5,thick]
    \tikzstyle{every node}=[minimum width=0pt, inner sep=2pt, circle]
        \draw (0:1) node[draw] (1) { };
        \draw (60:1) node[draw] (0) { };
        \draw (120:1) node[draw] (2) { };
        \draw (180:1) node[draw] (3) { };
        \draw (240:1) node[draw] (5) { };
        \draw (300:1) node[draw] (4) { };
        \draw  (0) edge (1);
        \draw  (0) edge (2);
        \draw  (0) edge (3);
        \draw  (0) edge (4);
        \draw  (1) edge (3);
        \draw  (1) edge (4);
        \draw  (2) edge (3);
        \draw  (3) edge (5);
        \draw  (4) edge (5);
    \draw (0,-1.5) node () { \text{\texttt{EvoW}} };
    \end{tikzpicture}
&
    \begin{tikzpicture}[scale=.5,thick]
    \tikzstyle{every node}=[minimum width=0pt, inner sep=2pt, circle]
        \draw (0:1) node[draw] (1) { };
        \draw (60:1) node[draw] (0) { };
        \draw (120:1) node[draw] (2) { };
        \draw (180:1) node[draw] (3) { };
        \draw (240:1) node[draw] (5) { };
        \draw (300:1) node[draw] (4) { };
        \draw  (0) edge (1);
        \draw  (0) edge (2);
        \draw  (0) edge (3);
        \draw  (0) edge (4);
        \draw  (1) edge (3);
        \draw  (1) edge (4);
        \draw  (2) edge (3);
        \draw  (3) edge (4);
        \draw  (3) edge (5);
        \draw  (4) edge (5);
    \draw (0,-1.5) node () { \text{\texttt{EvsW}} };
    \end{tikzpicture}
&
    \begin{tikzpicture}[scale=.5,thick]
    \tikzstyle{every node}=[minimum width=0pt, inner sep=2pt, circle]
        \draw (0:1) node[draw] (0) { };
        \draw (60:1) node[draw] (1) { };
        \draw (120:1) node[draw] (2) { };
        \draw (180:1) node[draw] (3) { };
        \draw (240:1) node[draw] (5) { };
        \draw (300:1) node[draw] (4) { };
        \draw  (0) edge (1);
        \draw  (0) edge (2);
        \draw  (0) edge (4);
        \draw  (1) edge (2);
        \draw  (1) edge (3);
        \draw  (1) edge (4);
        \draw  (2) edge (3);
        \draw  (2) edge (5);
        \draw  (3) edge (5);
        \draw  (4) edge (5);
    \draw (0,-1.5) node () { \text{\texttt{Ezow}} };
    \end{tikzpicture}
&
    \begin{tikzpicture}[scale=.5,thick]
    \tikzstyle{every node}=[minimum width=0pt, inner sep=2pt, circle]
        \draw (0:1) node[draw] (0) { };
        \draw (60:1) node[draw] (1) { };
        \draw (120:1) node[draw] (2) { };
        \draw (180:1) node[draw] (3) { };
        \draw (240:1) node[draw] (5) { };
        \draw (300:1) node[draw] (4) { };
        \draw  (0) edge (1);
        \draw  (0) edge (2);
        \draw  (0) edge (4);
        \draw  (1) edge (2);
        \draw  (1) edge (3);
        \draw  (1) edge (4);
        \draw  (2) edge (3);
        \draw  (2) edge (4);
        \draw  (2) edge (5);
        \draw  (3) edge (4);
        \draw  (3) edge (5);
        \draw  (4) edge (5);
    \draw (0,-1.5) node () { \text{\texttt{Ez{\char`\{}w}} };
    \end{tikzpicture}
&
    \begin{tikzpicture}[scale=.5,thick]
    \tikzstyle{every node}=[minimum width=0pt, inner sep=2pt, circle]
        \draw (0:1) node[draw] (0) { };
        \draw (60:1) node[draw] (1) { };
        \draw (120:1) node[draw] (2) { };
        \draw (180:1) node[draw] (3) { };
        \draw (240:1) node[draw] (4) { };
        \draw (300:1) node[draw] (5) { };
        \draw  (0) edge (1);
        \draw  (0) edge (2);
        \draw  (0) edge (3);
        \draw  (0) edge (4);
        \draw  (0) edge (5);
        \draw  (1) edge (2);
        \draw  (1) edge (3);
        \draw  (1) edge (4);
        \draw  (1) edge (5);
        \draw  (2) edge (3);
        \draw  (2) edge (4);
        \draw  (2) edge (5);
        \draw  (3) edge (4);
        \draw  (3) edge (5);
        \draw  (4) edge (5);
    \draw (0,-1.5) node () { \text{\texttt{E{\char`\~}{\char`\~}w}} };
    \end{tikzpicture}
&
    \begin{tikzpicture}[scale=.5,thick]
    \tikzstyle{every node}=[minimum width=0pt, inner sep=2pt, circle]
        \draw (0:1) node[draw] (0) { };
        \draw (60:1) node[draw] (3) { };
        \draw (120:1) node[draw] (2) { };
        \draw (180:1) node[draw] (1) { };
        \draw (240:1) node[draw] (4) { };
        \draw (300:1) node[draw] (5) { };
        \draw  (0) edge (1);
        \draw  (0) edge (2);
        \draw  (0) edge (3);
        \draw  (0) edge (5);
        \draw  (1) edge (2);
        \draw  (1) edge (3);
        \draw  (1) edge (4);
        \draw  (2) edge (3);
        \draw  (2) edge (4);
        \draw  (3) edge (5);
        \draw  (4) edge (5);
    \draw (0,-1.5) node () { \text{\texttt{E{\char`\~}YW}} };
    \end{tikzpicture}
\\
&
&
&
    \begin{tikzpicture}[scale=.5,thick]
    \tikzstyle{every node}=[minimum width=0pt, inner sep=2pt, circle]
        \draw (0:1) node[draw] (0) { };
        \draw (60:1) node[draw] (1) { };
        \draw (120:1) node[draw] (2) { };
        \draw (180:1) node[draw] (3) { };
        \draw (240:1) node[draw] (4) { };
        \draw (300:1) node[draw] (5) { };
        \draw  (0) edge (1);
        \draw  (0) edge (2);
        \draw  (0) edge (3);
        \draw  (0) edge (4);
        \draw  (0) edge (5);
        \draw  (1) edge (2);
        \draw  (1) edge (3);
        \draw  (1) edge (4);
        \draw  (2) edge (3);
        \draw  (2) edge (4);
        \draw  (3) edge (4);
        \draw  (3) edge (5);
        \draw  (4) edge (5);
    \draw (0,-1.5) node () { \text{\texttt{E{\char`\~}{\char`\}}W}} };
    \end{tikzpicture}
    &
    \\
\end{tabular}
\end{center}
\caption{Some forbidden for $\mathcal{S}_{\leq 4}$.}
\label{fig:forb4}
\end{figure}


The problem of characterizing graphs in $\mathcal{S}_{\leq4}$ is not straightforward.
However, it is interesting that if $G\in\mathcal{S}_{\leq4}$, then any graph obtained by replacing its vertices by stable sets will be also in $\mathcal{S}_{\leq4}$.
This will be shown next.

\begin{lemma}
For any $\mathbf{d}\in \mathbb{N}^{V}$, the non-zero invariant factors of $A(G)$ are equal to the non-zero invariant factors of $A(G^{\mathbf{d}})$.
\end{lemma}
\begin{proof}
Given $u\in V$.
Let $G^u$ denote the graph obtained after duplicating vertex $u$.
Since the adjacency matrix of $G^u$ is equivalent to
$
\begin{bmatrix}
A(G) & {\bf 0}^T\\
{\bf 0} & 0 \\
\end{bmatrix}
$,
then the non-zero invariant factors of $A(G)$ and $A(G^u)$ are the same.
From which the result follows.
\end{proof}

Previous lemma help us in computing the Smith normal form of the adjacency matrix of graphs with duplicated vertices.
In particular, it bound the number of non-zero invariant factors.

\begin{corollary}
Let $G$ be a graph in $\mathcal{S}_{\leq k}$, then $G^{\mathbf{d}} \in \mathcal{S}_{\leq k}$ for any $\mathbf{d}\in \mathbb{N}^{V}$.
\end{corollary}

\section*{Acknowledgement}
This research was initiated during GRWC 2018 (Graduate Research Workshop in Combinatorics) at Iowa State University.
Carlos A. Alfaro was supported by SNI and Ralihe R. Villagr\'an was supported by CONACyT.

\end{document}